\documentclass[12pt]{amsart}
\usepackage{amsmath}
\usepackage{amsfonts,mathtools}
\usepackage{amssymb,color}
\usepackage[hidelinks]{hyperref}
\usepackage{amsthm,doi}
\usepackage{thmtools,bbm}
\usepackage{tikz,tikz-cd,color}
\usepackage{pgfplots}
\usepackage{enumitem,url}
\usepackage{soul}
\pgfplotsset{compat=1.15}

\newcommand{\cu}{\kappa_{\partial\Omega,\eta}}
\newcommand{\til}{\widetilde{T}}

\newcommand{\N}{\mathbb{N}}
\newcommand{\C}{\mathbb{C}}
\newcommand{\R}{\mathbb{R}}

\newcommand{\pp}{\mathbf{P}}
\newcommand{\hh}{\mathcal{H}^1}
\newcommand{\ww}{\mathcal{W}}

\newcommand{\dd}{d}

\DeclareMathOperator{\re}{Re}
\DeclareMathOperator{\im}{Im}
\DeclareMathOperator{\tr}{trace}
\DeclareMathOperator{\erfc}{erfc}

\newtheorem{lemma}{Lemma}[section]
\newtheorem{theorem}[lemma]{Theorem}
\newtheorem{coro}[lemma]{Corollary}
\newtheorem{prop}[lemma]{Proposition}

\newtheorem{definition}[lemma]{Definition}

\theoremstyle{remark}
\newtheorem{rema}[lemma]{Remark}

\numberwithin{equation}{section}

\makeatletter
\@namedef{subjclassname@2020}{\textup{2020} Mathematics Subject Classification}
\makeatother

\author{Felipe Marceca}
\address[F. Marceca]{Department of Mathematics \\  King’s College London \\ United Kingdom}
\email{felipe.marceca@kcl.ac.uk}

\author{Jos\'e Luis Romero}
\address[J. L. Romero]{Faculty of Mathematics \\
	University of Vienna \\
	Oskar-Morgenstern-Platz 1 \\
	A-1090 Vienna, Austria \\and
	Acoustics Research Institute\\ Austrian Academy of Sciences\\Dr. Ignaz Seipel-Platz 2,	AT-1010 Vienna, Austria}
\email{jose.luis.romero@univie.ac.at}

\title[Improved discrepancy for the planar Coulomb gas]{Improved discrepancy for the planar Coulomb gas at low temperatures}

\thanks{J. L. R. and F. M. gratefully acknowledge support from the Austrian Science Fund (FWF): 10.55776/Y1199. \\ F. M. was also supported by the EPSRC: NIA EP/V002449/1.}



\begin{document}
\begin{abstract}
We study the planar Coulomb gas in the regime where the inverse temperature $\beta_n$ grows at least logarithmically with respect to the number of particles $n$ (freezing regime, $\beta_n\gtrsim \log n$). We show that, almost surely for large $n$, the discrepancy between the number of particles in any microscopic region and their expected value (given with adequate
precision by the equilibrium measure) is, up to log factors, of the order of the perimeter of the observation window. The estimates are valid throughout the whole droplet (the region where the particles accumulate), and are particularly interesting near the boundary, while in the bulk they offer technical improvements over known results.

Our work builds on recent results on equidistribution at low temperatures and improves on them by providing refined spectral asymptotics for certain Toeplitz operators on the range of the erfc-kernel (sometimes called Faddeeva or plasma dispersion kernel).
\end{abstract}

\maketitle

\section{Introduction}

We study the planar Coulomb gas with respect to an \emph{external potential} $Q$ and \emph{inverse temperature} $\beta>0$.
The potential is a lower semicontinuous function $Q:\C\to \R\cup\{+\infty\}$ that is finite on some set of positive capacity and satisfies
\begin{align}
	\label{qgro}
	\liminf_{z\to\infty} \frac{Q(z)}{2\log|z|}>1.
\end{align}
To a point configuration $\{z_j\}_{j=1}^n$ we associate the Hamiltonian
\begin{align}\label{eq_ham}
H_n=\sum_{j\neq k} \log \frac{1}{|z_j-z_k|} + n \sum_{j=1}^n Q(z_j),
\end{align}
and consider the Boltzmann-Gibbs probability measure on $\C^n$
\begin{align}\label{bg}
d \pp_n^\beta = \frac{1}{Z_n^\beta} e^{-\beta H_n} dA(z_1)\ldots dA(z_n),
\end{align}
where $Z_n^\beta$ is a normalizing constant and $dA(z_j)$ denotes the Lebesgue on $\C$ divided by $\pi$.

We think of a random configuration $\{z_j\}_{j=1}^n$ drawn from \eqref{bg} as consisting of identical repelling point charges, which are prevented from dispersion by the presence of the rescaled field $nQ$. In the formal zero temperature limit $\beta=\infty$ we let $\{z_j\}_{j=1}^n$ minimize $H_n$ and speak of a \emph{Fekete configuration}. We shall be interested in large but finite $\beta$.

\emph{The equilibrium measure} $\mu$ associated with $Q$ \cite{SaTo} is the unique compactly supported Borel probability measure minimizing the logarithmic $Q$-energy given by
\[I_Q(\nu)=\int_{\C^2} \log \frac{1}{|z-w|}d\nu(z)d\nu(w) + \int_\C Q(z)d\nu(z).\]
The support $S$ of $\mu$ is known as the \emph{droplet}. Provided that $Q$ is $C^2$ on a neighborhood of $S$, the equilibrium measure is
\begin{align}\label{eq_equil}
d\mu =\chi_S \Delta Q dA,
\end{align}
where $\Delta=1/4(\partial_{xx}+\partial_{yy})$ is one quarter of the standard Laplacian.

Point configurations $\{z_j\}_{j=1}^n$ drawn from $\pp_n^\beta$ tend to follow the equilibrium measure. To make this statement precise, for $\{z_j\}_{j=1}^n$ and $\Omega\subseteq \C$ we write $\#[\Omega]=\#(\Omega\cap \{z_j\}_{j=1}^n)$ and define the 1-point intensity $R_n^\beta$ at $z\in \C$ by
\begin{align}\label{eqa}
R_n^\beta(z)=\lim_{\varepsilon\to 0} \frac{\mathbf{E}_n^\beta\big(\#[B(z,\varepsilon)]\big)}{\varepsilon^2}.
\end{align}
The first intensity $R_n^\beta(z)$ describes the expected number of particles per unit area at $z$ and satisfies
\[\frac{1}{n}R_n^\beta dA \xrightarrow{n\to \infty} d\mu,\]
in the weak sense of measures, even for $n$-dependent inverse temperatures $\beta=\beta_n$ as soon as they are bounded below --- see \cite[Appendix A]{Am} and \cite{Jo, HeMa, MR1606719}.

We will be concerned with the \emph{low temperature regime}, where inverse temperatures grow with the number of particles $n$ as
\begin{align}\label{eq_lt}
\beta=\beta_n\ge c\log n,
\end{align}
with $c>0$ a constant. In this regime, introduced in \cite{amcmp} and further studied in \cite{amro20, ammaro}, particles are expected to follow $\mu$ in a much more rigid way than that indicated by the first order statistic \eqref{eqa}. We investigate this at the microscopic scale, of order $n^{-1/2}$.

We shall make the following assumptions on $Q$, which are essentially the same as in \cite{amro20}:
\begin{enumerate}[label=(\roman*)]
	\item \label{c1} $Q:\C\to \R\cup\{+\infty\}$ is lower semicontinuous, finite on some set of positive capacity 
	and satisfies \eqref{qgro};
	\item \label{c5} $Q$ is real-analytic in a neighborhood of $S$;
	\item \label{cau} $\Delta Q >0$ in a neighborhood of $S$;
	\item \label{c2} The boundary $\partial S$ has no singularities;
	\item \label{c4} $S^*=S$ where $S^*$ is the coincidence set for the obstacle problem associated with $Q$. Concretely, this means that for $z\in S^c$,
	\[Q(z)>\sup\{f(z): \ f\in \mathcal{F}_Q\},\]
	where $\mathcal{F}_Q$ denotes the class of subharmonic functions on $\C$ that are everywhere $\le Q$ and satisfy $f(z)\le \log |z|^2
	+ O(1)$ as $|z|\to\infty$.
	\item \label{c7} $S$ is connected.
\end{enumerate}
The assumptions make sense in light of Sakai's regularity theory \cite{MR1097025}. While Conditions \ref{c1}, \ref{c5} and \ref{cau} imply that $\partial S$ is composed of finitely many real-analytic simple curves with at most finitely many singular points, Condition \ref{c2} asserts that such singularities do not occur. (See \cite[Lemma 3.1]{AmCr} for Sakai's theory adapted to the current setting, and also \cite{AmKaMaWe,MR3454377, HeSh} for additional background.) In contrast, potentials that are merely $C^\infty$ could give rise to highly irregular droplets \cite[Theorem 11.1]{MR679313}. Condition \ref{c4} is somehow technical; in the parlance of \cite{HeMa} it means that $Q$ has no shallow points. Such assumption could be avoided by redefining the potential $Q$ to be $+\infty$ outside a small enough neighborhood of the droplet. We refer the reader to \cite{AmKaMaWe} for a discussion on these topics, and many literature references.

To formulate our results, we consider a sequence of inverse temperatures $\beta_n>0$
and a family of point configurations $\{z_j^{(n)}:1\le j\le n, n\in \N \}\subseteq \C$ drawn from the product measure $\mathbf{P}=\prod_{n\in\N }\mathbf{P}_n^{\beta_n}$. Almost sure events are also understood with respect to $\mathbf{P}$. To alleviate the notation, we write $z_j$ instead of $z_j^{(n)}$.

As in \cite{AmOC, amro20}, we investigate the microscopy statistics of the sample $\{z_j\}_{j=1}^n$ by means of a sequence $\{p_n\}_{n=1}^\infty \subset \mathbb{C}$ of \emph{zooming points}, with $p_n \to p_*$. Specifically,
\cite{amro20} investigates the Beurling-Landau densities
\begin{align}\label{eq_ber}
\limsup_{L \to \infty} \limsup_{n \to \infty} \frac{\#\left[B(p_n,L/\sqrt{n})\right]}{L^2},
\qquad
\liminf_{L \to \infty} \liminf_{n \to \infty} \frac{\#\left[B(p_n,L/\sqrt{n})\right]}{L^2},
\end{align}
and shows that under \eqref{eq_lt}, they approach the equilibrium density --- which is $\Delta Q(p_*)$ if $\{p_n\}_n$ zooms into the so-called bulk and $\tfrac12 \Delta Q(p_*)$ near the boundary. These statements hold almost surely and simultaneously for all suitable choices of zooming points \cite{amro20}. The goal of this article is to refine such statements, specially in the boundary case.  On the other hand, if \eqref{eq_lt} does not hold (for large $n$), then we do \emph{not} expect the Beurling-Landau densities \eqref{eq_ber} to be almost surely positive and finite --- and, indeed, in dimension 1 and for the classical Gaussian potential such intuition was confirmed in \cite{ammaro} by resorting to \cite{LeRi, HoVa, VaVi}.

\subsection{Discrepancies near the boundary}
In the low temperature regime \eqref{eq_lt}, the Coulomb gas is almost surely equidistributed in a technical sense measured by Beurling-Landau densities \eqref{eq_ber} \cite{amro20}. The proof of that result provides the following discrepancy estimate near the boundary of the droplet. Under \eqref{eq_lt} and Conditions \ref{c1} $\ldots$ \ref{c7}, \cite[Proposition~1.3]{amro20} shows that
the following holds almost surely: if $p_n \to p_*$ is a sequence of zooming points approaching the boundary as
\begin{align*}
	\limsup_{n\to\infty}\sqrt{n}\, d(p_n, \partial S) \leq M < \infty,
\end{align*}
with $M \geq 1$, then
\begin{align}\label{ebou}
	\limsup_{n\to\infty}\left|\#\left[B(p_n,L/\sqrt{n})\right]- \tfrac{1}{2}\Delta Q(p_*)  L^2 \right| \le CM L^{5/3} \log L.
\end{align}
Here $C$ is a constant that depends on the potential $Q$ and the constant $c$ in \eqref{eq_lt}. \footnote{In \cite{amro20} the dependence on $M$ is left implicit, while here, for the sake of comparison, we quote the dependence that follows from the proof.}

As a main result, we show that in the above situation the following bound holds:
\begin{align}\label{eboupp}
	\limsup_{n\to\infty}\left|\#\left[B(p_n,L/\sqrt{n})\right]- \tfrac{1}{2}\Delta Q(p_*)  L^2 \right| \le CM L\sqrt{\log L}\log\log L.
\end{align}
This in turn follows from the following result concerning general compact observation domains $\Omega \subseteq \mathbb{C}$ in lieu of disks. We denote the usual Lebesgue measure by $|\cdot |$ and the 1-dimensional Hausdorff measure by $\hh$.

\begin{theorem}
	\label{cdis}
	Assume that $Q$ satisfies Conditions \ref{c1} to \ref{c7}, $\beta_n\ge c \log n$ with $c>0$, and $\Omega\subseteq \C$ is a  compact domain  with measure $|\Omega|\ge 2$ and connected boundary.
	Then there exists a deterministic constant $C=C(c,Q)$ such that almost surely,
	\begin{multline}\label{eq_a2}
		\limsup_{n\to\infty}\sup_{p\in \C} \left|\#\left[p+\Omega/\sqrt{n}\right]- \tfrac{n}{\pi} \Delta Q(p)  {\big|(p+\Omega/\sqrt{n})\cap S\big|} \right| 
		\\ \le C \hh(\partial\Omega)\sqrt{\log  |\Omega|}   \log(1+\log |\Omega|).
	\end{multline}
(Here, although $\Delta Q$ is only defined in a neighborhood $U$ given by Condition \ref{c5}, for $p\in U^c$ we interpret $\Delta Q(p)  |(p+\Omega/\sqrt{n})\cap S|$ as 0.) 
\end{theorem}
Equivalently, as the events in question are independent for different values of $n$, by the Borel-Cantelli lemma, the almost sure validity of \eqref{eq_a2} means that the failure probabilities
\begin{multline*}
p_n := \mathbb{P} \Big[ \sup_{p\in \C} \left|\#\left[p+\Omega/\sqrt{n}\right]- \tfrac{n}{\pi} \Delta Q(p)  {\big|(p+\Omega/\sqrt{n})\cap S\big|} \right| 
\\> C \hh(\partial\Omega)\sqrt{\log  |\Omega|}   \log(1+\log |\Omega|)
\Big]
\end{multline*}
are summable: $\sum_{n \geq 1} p_n <\infty$.

Some remarks are in order. First, $\pi^{-1}\Delta Q(p)|(p+\Omega/\sqrt{n})\cap S|$ is an approximation of $\mu((p+\Omega/\sqrt{n})\cap S)$, so interpreting it as $0$ when $p$ is away from a neighborhood of $S$ makes sense since in this case $(p+\Omega/\sqrt{n})\cap S=\varnothing$ for all sufficiently large $n$. 

Second, we note that \eqref{eq_a2} is trivial when $p$ is microscopically far away from $S$ in the sense that $d(p, S) \geq M /\sqrt{n}$ for sufficiently large $M$, since in this case the observation window finds no points with high probability \cite{Am} and therefore the magnitudes being compared in the left-hand side of \eqref{eq_a2} are both eventually 0.

Third, as an estimate concerning general observation windows $\Omega$, the discrepancy bound \eqref{eq_a2} is, up to logarithmic factors, probably optimal, since it almost matches the discrepancies of lattices measured on cubes. On the other hand, it is conceivable that for particular observation domains  even stronger estimates may hold, as is the case with lattice configurations and disks \cite{hardy1915expression}.
In Section \ref{sbou} we provide a version of Theorem \ref{cdis} for possibly non-connected observation windows.

While Theorem \ref{cdis} estimates the microscopic deviations of point statistics, it is also interesting to have a comparison with a limiting quantity as in \eqref{eboupp}. To include non-circular observation windows, we first introduce the \emph{signed distance} of $p \in \mathbb{C}$ to $\partial S$:
\begin{align}\label{edis}
	d^\pm(p,\partial S) = \begin{cases}
		\hphantom{-}d(p,\partial S) & \mbox{if }  p \in S 
		\\
		-d(p,\partial S) & \mbox{if } p \notin S
	\end{cases}.
\end{align}
If $p_n \to p \in \partial S$, $e^{i\theta}$ is the outer unit normal of $p$ at $\partial S$ and 
\[l=\lim_{n\to\infty} \sqrt{n} \cdot d^\pm(p_n,\partial S),\]
then, as $n$ grows, the set $(p_n+\Omega/\sqrt{n})\cap S$ approaches $p+(\Omega \cap \{z: \re(z e^{-i\theta}) \leq l\})/\sqrt{n}$. Hence, as an application of Theorem \ref{cdis}: 
\begin{multline}\label{more}
	\limsup_{n\to\infty}\left|\#\left[p_n+\Omega/\sqrt{n}\right]- \Delta Q(p) \frac{|\Omega \cap \{z: \re(z e^{-i\theta}) \leq l\}|}{\pi} \right| 
	\\ \le C \hh(\partial\Omega)\sqrt{\log  |\Omega|}   \log(1+\log |\Omega|).
\end{multline}
See Theorem~\ref{tdis2} for details. 

We remark that, although not discussed explicitly, all results hold also for the formal limit $\beta_n = \infty$ (Fekete sets) where particles are chosen as minimizers of the Hamiltonian \eqref{eq_ham}.

\subsection{Discrepancies in the bulk}
While Theorem \ref{cdis} concerns the whole droplet, our main motivation is to study discrepancies near its boundary. As we now discuss, slightly stronger estimates hold away from the boundary. For simplicity, in this case we assume that $Q$ is finite and $\mathcal{C}^2$ on the whole complex plane (denoted $Q \in \mathcal{C}^2$).

Bulk discrepancy estimates for arbitrary inverse temperatures $\beta$, general dimension, and point statistics computed on cubes are proved in \cite{ArSe}. When applied to the planar case and the low temperature regime \eqref{eq_lt}, \cite[Theorem~1]{ArSe} combined with a covering argument gives the following: there are constants $M=M(Q)$ and $C=C(c,Q)$ such that almost surely, if $\{p_n\}_{n=1}^\infty \subset \mathbb{C}$ satisfies the bulk regime condition
\begin{align}\label{eq_arsebulk}
d(p_n, S^c)\ge M n^{-1/4},
\end{align}
then
\begin{align}\label{eq_bb}
	\limsup_{n\to\infty}\left|\#\left[Q(p_n,L/\sqrt{n})\right]- \Delta Q(p_*)  \frac{L^2}{\pi} \right| \le C L,
\end{align}
where $Q(x,r)$ is a square of center $x$ and length $r$. With respect to disks of radius $L$, and other general observation windows dilated by a factor of $L$, a Whitney decomposition argument can be combined with \cite[Theorem~1]{ArSe} to obtain a bound similar to \eqref{eq_bb} with an extra $\log L$ factor on the right-hand side. Similar estimates for Fekete sets ($\beta=\infty$) go back to
\cite{MR3373044}.

On the other hand, as a technical step towards the equidistribution result for Beurling densities, \cite[Proposition~1.3]{amro20} implies that
under the less stringent bulk assumption
\begin{align*}
	\lim_{n\to \infty} n^{1/2} d(p_n, S^c) = \infty
\end{align*}
a bound similar to \eqref{eq_bb} holds with $L^{5/3}$ in lieu of $L$ (and point statistics computed on disks.) 

In this article we prove that $O(L)$ bulk discrepancy, as in \eqref{eq_bb}, holds under the rather weak bulk assumption $d(p_n,S^c)\ge M n^{-1/2} \log n$, for large $M$, and for general observation windows.

\begin{theorem}	\label{cdisbul}
	Assume that $Q\in \mathcal{C}^2$ and satisfies Conditions \ref{c1} to \ref{c4}, $\beta_n\ge c \log n$ with $c>0$, and $\Omega\subseteq \C$ is a compact domain with measure $|\Omega|\ge 1$ and connected boundary. Then there exist deterministic constants $M=M(Q)$ and $C=C(c,Q)$ such that, 
	almost surely,
	\begin{align}\label{eq_disbul}
		\limsup_{n\to\infty}\sup_{\substack{p\in S,\\ d(p,S^c)\ge M n^{-1/2}\log n}}\left|\#\left[p+\Omega/\sqrt{n}\right]-  \Delta Q (p)\frac{|\Omega|}{\pi}
		\right| \le C \hh(\partial\Omega).
	\end{align}
\end{theorem}
We also provide a more quantitative version of Theorem \ref{cdisbul} which also concerns more general observation windows in Theorem \ref{tdisbul2}.

Let us mention that the recent remarkable \emph{overcrowding estimates} from \cite{Th} (valid for general dimension and temperature) have certain consequences for discrepancy statistics, which are moderately related to Theorem \ref{cdisbul}. Specifically, \cite[Theorem 5]{Th} implies a \emph{semi-discrepancy} estimate
\begin{align}\label{eq_th}
\#[B(p_*,L/\sqrt{n})] - \Delta(Q)(p_*) L^2 \leq C L^{4/3} \log L
\end{align}
valid for each $p_* \in \mathbb{C}$ with high probability (note the absence of absolute values in \eqref{eq_th}). As opposed to \cite[Theorem 1]{ArSe}, failure probabilities in \cite[Theorem 5]{Th} do not depend on $n$ and $\beta$, and thus do not exhibit a gain in the low temperature regime \eqref{eq_lt} that would lead to a uniform almost sure bound as in Theorem \ref{cdisbul}. Note also that even if the probability bounds in \cite[Theorem 5]{Th} could be improved to exhibit such a gain in the low temperature regime (and we suspect they can), the estimate would only be meaningful as a semi-discrepancy under the (quite sharp) bulk assumption $n^{1/2} d(p,S^c) \gtrsim 1$. Outside this regime \eqref{ebou} --- or Theorem \ref{cdis} --- shows that for large $L$ the left-hand side of \eqref{eq_th} is with high probability a negative number of order $L^2$, as, indeed,  $\tfrac12 \Delta Q(p_*) L^2$ is a more accurate approximation of the point statistic in question. A similar discussion applies to the discrepancy estimates in \cite[Theorem 6]{Th}, which are of the order $L^{6/5} \log L$ and 
valid under the bulk condition \eqref{eq_arsebulk}.

Finally, we mention that the regularity assumptions on $Q$ are excessive for the bulk regime, while \cite{ArSe, Th} require minimal regularity. The results we rely on are however known to admit regularity improvements at the cost of more technical formulations
\cite[Appendix A.2]{AmHeMa2}, specially since we do not use them in their full force. Similarly, the assumption that $Q\in \mathcal{C}^2$ in Theorem \ref{cdisbul} can be dropped, as regularity is only needed near the droplet; see Remark \ref{rem_c2}.

\subsection{Technical overview and methods}\label{sec_methods}
Our proof of Theorem \ref{cdis} elaborates on the equidistribution result in \cite{amro20} and, as we now discuss, refines a key step to obtain the announced discrepancy estimates. The outset of the argument uses ideas from sampling theory, which also play a similar role in \cite{AmOC, MR3831027, LeCe, ammaro}, and we now briefly review. The basic intuition goes back to the theory of orthogonal polynomials on the real line: as $\beta\to\infty$, Coulomb samples are expected to approach Fekete configurations, which provide \emph{quadrature rules} for weighted polynomials \cite{Meh,Ism}.

In our case, the central object is the space of \emph{planar weighted polynomials},
\begin{align*}
\ww_n=\left\{f=q e^{-nQ(z)/2}: \ \deg q \le n-1\right\}\subseteq L^2(\C,dA),
\end{align*}
and its \emph{reproducing kernel} $K_n$, which represents the orthognal projection $P^{(n)}:L^2(\C,dA) \to \ww_n$. As shown in \cite[Theorem 5.1]{amro20}, under Conditions \ref{c1} to \ref{c7}, a random sample of the Coulomb gas, taken in the low temperature regime \eqref{eq_lt}, solves the following \emph{sampling and interpolation problems}:
\begin{itemize}
\item[(i)]: (`Sampling'). For $0< \rho <1$,
\begin{align}\label{eq_s1}
\int_{S + B_{M/\sqrt{n}}(0)} |f|^2 \, dA \leq \frac{C}{n(1-\rho)^2} \sum_{j=1}^n |f(\zeta_j)|^2, \qquad f \in \ww_{\rho n}.
\end{align}
\item[(ii)]: (`Interpolation'). For $\rho >1$,
given $\{a_j\}_{j=1}^n \subset \mathbb{C}$, there exists $f \in \ww_{\rho n}$
such that
\begin{align}\label{eq_s2}
\begin{aligned}
	&f(z_j)=a_j, \qquad j=1,\ldots,n, \mbox{ and}
	\\
	&\int_{\mathbb{C}} |f|^2 \, dA \leq \frac{C}{n(\rho-1)^2} \sum_{j=1}^n |a_j|^2.
\end{aligned}
\end{align}
\end{itemize}
Here $C$ and $M$ are deterministic constants, $\rho n$ is assumed to be an integer, and the properties hold for all sufficiently large $n$, depending on the allowed failure probability. It is important to note that the estimate \eqref{eq_s1} holds on the subspace $\ww_{\rho n} \subset \ww_n$ with $\rho <1$ while \eqref{eq_s2} provides an interpolant $f$ belonging to the larger space $\ww_{\rho n} \supset \ww_n$ with $\rho>1$.

The sampling and interpolation estimates allow us to invoke and aptly adapt a general principle going back to Landau's work on bandlimited functions with disconnected spectra \cite{la67}: Under \eqref{eq_s1} or \eqref{eq_s2}, the point statistic $\# \Omega_n$,
associated with a microscopic observation domain $\Omega_n = \Omega/\sqrt{n} + p_n$ scanning the boundary of the droplet,
can be asymptotically compared to the eigenvalue counting function of the \emph{concentration operator} $P^{(\rho n)} M_{\Omega_{n}}: \ww_{\rho n} \to \ww_{\rho n}$, $f \mapsto P^{(\rho n)}(f \chi_{\Omega_{n}})$. More precisely, for a carefully chosen \emph{spectral threshold} $\delta \in (0,\frac12)$,
\begin{align}\label{eq_s3}
\begin{aligned}
  \# \Omega_{n} &\geq \#\{\lambda \in \sigma(P^{(\rho n)} M_{\Omega_{n}}): \lambda > {1-\delta}\} + E_{\delta,\rho,n}, \qquad \rho <1,
  \\
  \# \Omega_{n} &\leq \#\{\lambda \in \sigma(P^{(\rho n)} M_{\Omega_{n}}): \lambda > {\delta}\} + E_{\delta,\rho,n}, \qquad \rho >1,
\end{aligned}
\end{align}
where $E_{\delta,\rho,n}$ are suitably bounded error terms. The choice of the spectral threshold $\delta$ is dictated by the bounds (stability margins) in the sampling \eqref{eq_s1} and interpolation inequalities \eqref{eq_s2}. The adequate choice satisfies
\begin{align}\label{eq_delta}
	\delta \asymp |1-\rho|^{2} \in (0,1/2).
\end{align}
The concentration operator $P^{(\rho n)} M_{\Omega_{n}}$ is a \emph{Toeplitz operator} with an indicator function as symbol, and, by
\eqref{eq_s3}, its limiting spectral profile tightly describes the point statistic $\# \Omega_{n}$.
As $n \to \infty$, the microscopic rescaling of the reproducing kernel $K_n$ converges to a translated \emph{erfc-kernel} (also known as Faddeeva or plasma dispersion kernel) \cite{MR4366416}:
\begin{align}\label{eq_fl}
	F_l(z,w)=G(z,w)F(z+\overline{w}-2l),
\end{align}
where
\begin{align}\label{egin}
	&G(z,w)=e^{z\overline{w}-|z|^2/2-|w|^2/2},
	\\
	\label{eerf}
	&F(z)=\frac{1}{2}\erfc\left(\frac{z}{\sqrt{2}}\right)
	= \frac{1}{\sqrt{2\pi}} \int_{-\infty}^0 e^{-(z-t)^2/2} dt.
\end{align}
This key fact is leveraged to take a limit on \eqref{eq_s3}. In the limit, the error terms $E_{\delta,\rho,n}$ can be dominated by the observation perimeter $\hh(\partial{\Omega})$.

On the other hand, the limit of the eigenvalue counting function in \eqref{eq_s3} is the eigenvalue counting function of a \emph{concentration operator acting on the range-space of an erfc-kernel} \eqref{eq_fl}:
\begin{align}\label{eq_x3}
T_{\rho} f (z) = \int_{\sqrt{\rho} \tilde{\Omega}} f(w) F_{l}(z,w)\, dA(w),
\end{align}
where $l=l(\rho)$ and the limiting domain $\tilde{\Omega}$ is obtained from $\Omega$ after rescaling near $p_n$. The equidistribution estimates in \cite{amro20} are proved by using the following crude spectral asymptotics for the operator \eqref{eq_x3},
\begin{align}\label{eq_x1}
\# \{ \lambda \in \sigma(T_{\rho}): \lambda > \delta \} = \rho \cdot |\tilde{\Omega} | + \hh( \partial\tilde{\Omega}) \cdot \frac{1}{\delta} \cdot O(1), 
\end{align}
which follow by comparing the first two moments of $T_{\rho}$ (trace and Hilbert-Schmidt norm). Indeed, the calculation of the Beurling-Landau densities \eqref{eq_ber} in \cite{amro20} follows by letting $\Omega$ be a disk of growing radius, and letting $\rho \to 1$ at an adequate speed dictated by \eqref{eq_delta}. Moreover, choosing $\rho$ as an adequate function of $|\tilde \Omega|$ and $\hh(\partial \tilde{\Omega})$ one obtains the discrepancy bound \eqref{ebou}.

To obtain the announced discrepancy bound \eqref{eq_a2} we shall improve \eqref{eq_x1} to
 \begin{align}\label{eq_x2}
 	\# \{ \lambda \in \sigma(T_{\rho}): \lambda > \delta \} = \rho\cdot|\tilde{\Omega} | + \hh(\partial\tilde{\Omega}) \cdot \sqrt{\log(\tfrac{1}{\delta})} \cdot \log \log (\tfrac{1}{\delta}) \cdot O(1).
 \end{align}
Our recent work \cite{MR} provides similar spectral estimates for concentration operators on the range of the \emph{Ginibre kernel} \eqref{egin} and other rapidly decaying kernels. While these do not directly apply to the slowly decaying erfc-kernel, we develop an approximation argument based on the \emph{Bargmann transform} \cite{MR157250}. This transformation maps
``signals'' in $L^2(\mathbb{R})$ into their ``phase-space representations'' in the range-space of the Ginibre kernel (the weighted Bargmann-Fock space of analytic functions). Under this correspondence, the action of the erfc-kernel translates into the cut-off operation $L^2(\mathbb{R}) \ni f \mapsto f \cdot \chi_{(-\infty,0)} \in L^2((-\infty,0))$. This insight, which also plays a role in \cite{AmKaMa}, is exploited here to approximately compare the erfc Toeplitz operator \eqref{eq_x3} to a Ginibre concentration operator with a suitably modified domain.

Some remarks are in order. First, spectral estimates such as \eqref{eq_x2} are classical in operator theory and usually studied in more generality. What is special of \eqref{eq_x2} is the quality of the $O(1)$ term, which is independent of the spectral threshold $\delta$, and allows us to choose it as a function of the geometry of the domain $\Omega$. In contrast, more classical asymptotics for dilated domains \cite{MR593228} provide more precise spectral expansions than \eqref{eq_x2} at the cost of requiring that $\delta$ be kept fixed (which would preclude the application intended here).

Second, we remark that the use of general observation domains is demanded by our proof strategy. Even if we were only interested in computing discrepancies on cubes or disks, the geometry of the droplet and the procedure to compare erfc and Ginibre Toeplitz operators lead to more general domains. Thus, while the eigenvalues of Ginibre Toeplitz operators 
associated with disks can be computed in terms of certain special functions \cite{da88}, we rely essentially on the more general results from \cite{MR}.

Third, we mention our recent work on separation and discrepancy of $\beta$-ensembles on the real line \cite{ammaro}, which follows a similar strategy as here, though with substantial technical differences in its implementation. In particular, \cite{ammaro} relies on eigenvalue estimates for sinc kernels proved in \cite{BoJaKa,KaRoDa}, while the technical work in this article is in the relation between erfc and Ginibre Toeplitz operators, and the particulars of two-dimensional geometry.

Finally, we comment on the bulk regime. Away from the boundary of the droplet, simpler arguments are possible because of the fast decay of the reproducing kernel, which can be precisely estimated also for finite $n$ \cite{MR2430724,AmHeMa,MR3010273}. To prove Theorem \ref{cdisbul} we follow and slightly adapt the argument of \cite{LeCe} to bound the discrepancy of Fekete points on general fiber bundles (with no droplet). While the argument of \cite{LeCe} is inspired by Landau's method \cite{la67}, and particularly by \cite{NiOl}, it does not explicitly refer to sampling \eqref{eq_s1} and interpolation \eqref{eq_s2} and avoids the need for the auxiliary parameter $\rho$. This leads to the discrepancy estimate \eqref{eq_disbul}, which, in contrast to
\eqref{eq_a2}, has no log factors, and also to a more precise estimate applicable to finite $n$; see Theorem \ref{tdisbul2}.

\subsection{Further remarks}
Our work relies on the convergence of the reproducing kernel of the space of weighted polynomials towards the erfc-kernel, under microscopic rescalings near the boundary \cite{HeWe}. Similar results are available for other model ensembles of contemporary interest
besides the ones studied here, including those of \cite{AmChCr22}, which concern radially symmetric potentials and droplets with ring-shaped spectral gaps, and the lemniscate ensembles from \cite{BY}. Our results extend suitably to these settings. 

Similarly, the method considered in this article may be also useful in the weakly non-Hermitian regime (also known as almost-Hermitian) that addresses the interface between dimensions 1 and 2. There the limiting kernels in the bulk interpolate between the sine-kernel and the Ginibre kernel \cite{FKS,ACV,AmBy}. Their structure is similar to \eqref{eq_fl}, taking the form:
\[\widetilde{K}(z,w)=G(z,w)\widetilde{F}(-i(z-\overline{w})),\]
where for some $a>0$,
\[\widetilde{F}(z)=\frac{1}{\sqrt{2\pi}} \int_{-a}^{a} e^{-(z-t)^2/2} dt.\]
We expect that the spectral estimates presented here can be adapted to the Toeplitz operators associated with such kernels.

Finally, although not explicitly discussed, all our results apply also to the formal limit $\beta=\infty$ (Fekete configurations), as our analysis is based on the sampling and interpolation properties described in Section \ref{sec_methods}, and these also hold for Fekete points \cite{AmOC, amro20}. In fact, samples of the Coulomb gas at low temperature can be considered easier to manufacture proxies for Fekete configurations, that yet retain similar desirable properties. To the best of our knowledge, Theorem \ref{cdis} and other results concerning the boundary are novel even in the Fekete case, and thus improve on corresponding results in \cite{AmOC,MR3373044}.

\subsection*{Organization}
The connection of our problem to sampling theory, and the details on Landau's method described above can be found in \cite{amro20} and will not be repeated here. Rather, we quote certain technical results from \cite{amro20} and build on them. This paper is organized as follows.

In Section~\ref{spre} we establish notation and provide some auxiliary results. 

In Section~\ref{sspe} we derive spectral estimates for concentration operators associated with the erfc-kernel.

In Section~\ref{sbou} we prove our main result, Theorem~\ref{cdis}, and Theorem \ref{tdis2}, which gives \eqref{more}.

In Section~\ref{sbul} we prove the discrepancy result in the bulk, Theorem~\ref{cdisbul}, together with a more general version, Theorem \ref{tdisbul2}.

\subsection*{Acknowledgment}
The authors are very grateful to Yacin Ameur for his insightful comments.

\section{Preliminaries}\label{spre}

\subsection{Notation and basic facts}
We write $a\lesssim b$ whenever there is a constant $C>0$ depending possibly on $Q$ and $c$ (recall that $\beta_n\ge c \log n$) but not on other parameters, such that $a\le C b$.

Recall that $dA(z)$ is the differential of the area measure divided by $\pi$. To shorten expressions we sometimes write $dA(z_1,\ldots,z_n)$ instead of $dA(z_1)\ldots dA(z_n)$. We also let $|\cdot|$ denote the usual Lebesgue measure and $\hh$ the 1-dimensional Hausdorff measure. 

We will work with the reproducing kernel $K_n(z,w)$ of the space of weighted polynomials
\[\ww_n=\{f=q e^{-nQ(z)/2}: \ \deg q \le n-1\}\subseteq L^2(\C,dA).\] 
It is worth mentioning that this is the canonical correlation kernel for the determinantal case $\beta=1$ \cite{Fo}.

\subsection{Scaling limits}\label{sec_sc}
\begin{definition}[Microscopic direction and position]
	\label{dlim}
	Let $\{p_n: n \geq 1\} \subset \mathbb{C}$ with
	$p_n \to p \in S$.
	
	$\bullet$ If $p_n\to p\in \partial S$, let $\theta\in[0,2\pi)$ be such that $e^{i\theta}$ is the outer unit normal of $\partial S$ at $p$.
	
	Furthermore, note that for sufficiently large $n$, there is a unique $q_n\in \partial S$ that is closest to $p_n$. Let $\theta_n\in[0,2\pi)$ be such that $e^{i\theta_n}$ is the outer unit normal of $\partial S$ at $q_n$. Define also the \emph{microscopic position relative to the boundary} as
	\begin{align}\label{elim}
		l=\lim_{n\to\infty} \sqrt{ n} d^\pm(p_n,\partial S)=\lim_{n\to\infty} \sqrt{ n} (q_n-p_n)e^{-i\theta_n},
	\end{align}
	provided that the limit exists.
	
	$\bullet$ For $p_n\to p\in  S^\circ$, set $l=\infty$ and $\theta_n=\theta=0$.
\end{definition}
In the definition above, when $p \in S^\circ$, the parameters $\theta_n,\theta$ play no important role; they serve the purpose of treating the bulk and boundary regimes together.

\begin{lemma}\label{lext}
	Any sequence $p_{n_k}\to p\in S$ with $\sqrt{ n_k} d^\pm(p_{n_k},\partial S)\to l$ and $e^{i\theta_{n_k}}\to e^{i\theta}$ as in Definiton \ref{dlim}, can be extended to $p_{n}\to p\in S$ with $\sqrt{ n} d^\pm(p_{n},\partial S)\to l$ and $e^{i \theta_{n}}\to e^{i\theta}$.
\end{lemma}
\begin{proof}
	First assume that $p\in\partial S$. For $k\in \N$ consider $q_{n_k}$ as in Definiton \ref{dlim} and for every $n_k<n<n_{k+1}$ define
	\[p_n=q_{n_k}-e^{i\theta_{n_k}}\sqrt{\tfrac{n_k}{n}}d^\pm(p_{n_k},\partial S).\]
	Notice that $p_n$ is the unique point in the segment $[p_{n_k},q_{n_k}]$ such that $\sqrt{ n} d^\pm(p_{n},\partial S)=\sqrt{ n_k} d^\pm(p_{n_k},\partial S)$. Moreover, $\theta_{n}=\theta_{n_k}$ and the conclusion follows.
	
	On the other hand, if $p\in  S^\circ$, we can simply take $p_n=p_{n_k}$ and $\theta_{n}=0$ for every $n_k<n<n_{k+1}$. It is easy to check that for this choice,
	\[\sqrt{ n} d^\pm(p_{n},\partial S)\to \infty. \qedhere\]
\end{proof}

For $p_n\to p\in S$, define the zooming map $\Gamma:\C\to\C$ similarly as in \eqref{elim} by
\[\Gamma(z)= \sqrt{n \Delta Q(p_n)} (z-p_n)e^{-i\theta_n},\]
where $\theta_n$ is defined as in Definition~\ref{dlim}.

The following fundamental result, which follows from \cite{MR4366416, HeWe} and \cite[Lemma~2]{AmKaMaWe}, describes the scaling limits of the reproducing kernels $K_n$ up to so-called cocycle factors. 
\begin{lemma}\label{lklim}
	Suppose that Q is real-analytic and strictly subharmonic in a neighborhood of the droplet. There exists a sequence of continuous functions $g_n: \mathbb{C} \to \mathbb{S}^1$ (cocycles) so that the following statements hold.
	\begin{itemize}
		\item[(i)] (`Bulk regime'): If $p_n\to p \in  S$ satisfies
		\begin{align*}
		\sqrt{n}\, d(p_n, S^c) \to \infty,
		\end{align*}
		then for every $z,w\in\C$,
		\[\lim_{n\to\infty} g_n(z)\overline{g_n(w)} \frac{K_n(\Gamma^{-1}(z),\Gamma^{-1}(w))}{n \Delta Q(p_n)}=G(z,w).\]
		
		\item[(ii)] (`Boundary regime'): If $S$ is connected, the boundary $\partial S$ is everywhere smooth, $p_n\to p \in \partial S$, and
		\begin{align*}
			\limsup_{n\to\infty}\sqrt{n}\, d(p_n, \partial S) < \infty,
		\end{align*}
		and the limit $l$ given by \eqref{elim} exists, then for every $z,w\in\C$,
		\[\lim_{n\to\infty} g_n(z)\overline{g_n(w)} \frac{K_n(\Gamma^{-1}(z),\Gamma^{-1}(w))}{n \Delta Q(p_n)}=F_{{\sqrt{\Delta Q(p)}} \cdot l}(z,w).\]
	\end{itemize}
Moreover, in both cases the convergence is uniform on compact subsets of $\C^2$.
\end{lemma}
While \cite{HeWe} provides higher order asymptotics for orthogonal polynomials, Lemma \ref{lklim} only needs first order ones, for which \cite[Section 5]{AmCr} can provide a shorter route. For more on scaling limits, we refer to \cite[Section 5.4]{byun2022progress} and the references therein.

\subsection{Boundary regularity} We will work with compact domains whose boundary satisfies the following geometric condition with respect to the one-dimensional Hausdorff measure $\mathcal{H}^1$.

\begin{definition}
	An $\hh$-measurable set $E \subseteq \R^2$ or $\C$ is said to be lower Ahlfors 1-regular (regular for short) at scale 
	$\eta>0$ if there exists a constant $\kappa>0$ such that
	\begin{align*}
		\hh\big(E \cap B_{r}(z) \big)\geq \kappa \cdot r, \qquad 0 < r \leq \eta, \quad z \in E.
	\end{align*}
	In this case, the largest possible constant $\kappa$ is denoted
	\begin{align*}
		\kappa_{E,\eta} = \inf_{0 < r \leq \eta} \inf_{z \in E}
		\frac{1}{r} \cdot \hh\big(E \cap B_{r}(z) \big).
	\end{align*}
\end{definition}

Note that if $\hh(E)>0$, by differentiation around a point of positive $\hh$-density, we always have
\begin{align}\label{eq_cubb}
	\kappa_{E,\eta} \leq 2.
\end{align}

Before proceeding with the study of discrepancy, we mention a few properties of regularity. Unlike in the high-dimensional case, regularity is a fairly general property in two dimensions. For completeness we include a proof of the following well-known fact.

\begin{lemma}\label{rrb}
	Let $E\subseteq \C$ be a compact connected set with more than one point. Then, $E$ is regular at scale $\eta=\hh(E)$ and $1\le \kappa_{E,\eta}\le 2$. 
\end{lemma}
\begin{proof}
	For $z\in E$, let $w\in E$ be a point at maximum distance from $z$. Notice that $0<\dd(z,w) \le \hh(E)$ since $E$ is connected and has more than one point. If $\dd(z,w)<r \le \hh(E)$, then $E\subseteq B_r(z)$ and
	\[ \hh(E\cap B_r(z))=\hh(E)\ge r.  \]
	On the other hand, assume $0<r\le \dd(z,w)$. Since $E$ is connected, $E\cap \partial B_s(z)\neq \varnothing$ for every $0<s\le \dd(z,w)$. Therefore, we have that $F(E\cap B_r(z))=[0,r)$ where $F(u):=\dd(z,u)$ has Lipschitz constant 1. So we also get
		\[ \hh(E\cap B_r(z))\ge r.  \]
	In conclusion, $E$ is regular at scale $\eta=\hh(E)$ and $1\le \kappa_{E,\eta}$. Finally, since $\hh(E)>0$ we have \eqref{eq_cubb}.
\end{proof}

 The following technical lemma will help us circumvent the fact that regularity of the boundary of a domain can be lost when cutting the domain with a half-plane.

\begin{lemma}
	\label{lcut}
	If a compact set $\Omega\subseteq\C$ has regular boundary at scale $\eta>0$ with parameter $\cu$, then the set $(\Omega\cap \{z: \re{z} \le 0\}) \cup \partial\Omega$ has regular boundary at scale $\eta>0$ with parameter $\kappa\ge \cu/2$.
\end{lemma}
\begin{proof}
An easy computation shows that
\[E:=\partial \big((\Omega\cap \{z: \re{z} \le 0\}) \cup \partial\Omega\big)=\partial \Omega \cup (\Omega^\circ\cap \{z:\re{z}=0\}).\]
Choose $w\in E$ and let us prove the regularity condition with the desired parameters. The case $w\in\partial\Omega$ is obvious, so assume $w\in \Omega^\circ\cap \{z:\re{z}=0\}$. If $0<r\le 2 \dd(w,\partial\Omega)$, then
	\[\hh( E \cap B_r(w))\ge \hh(\{z:\re{z}=0\} \cap B_{r/2}(w))= r \ge r\cu/2,\]
	where in the last step we used \eqref{eq_cubb}. 
	
	On the other hand, take $2 \dd(w,\partial\Omega)\le r \le \eta$ (if $\eta\le 2 \dd(w,\partial\Omega)$ we would be done, so we assume otherwise). So, there exists $u\in \partial \Omega$ such  that $2 \dd(w,u) \le r$ and therefore,
	\[\hh( E \cap B_r(w))\ge \hh(\partial \Omega \cap B_{r/2}(u))= r\cu /2. \qedhere\]
\end{proof}

We also need the following special case of \cite[Theorems 5 and 6]{Ca}.
\begin{prop}\label{prop_ca}
	There exists a universal constant $C>0$ such that for every compact set $E \subseteq \C$ that is regular at scale $\eta>0$, 
	\begin{align*}
		\hh \big(\{z:\ d(z,E)=r\} \big) \leq  \frac{C}{\kappa_{E,\eta}}  \cdot \hh(E) \cdot \Big(1+\frac{r}{\eta}\Big),
	\end{align*}
	for almost every $r>0$. In addition,
	$| \nabla d(z,E) | = 1$, for almost every $z \in \C$.
\end{prop}

\begin{coro}\label{coar}
	There exists a universal constant $C>0$ such that for every compact set $E \subseteq \C$ that is regular at scale $\eta>0$ and every $r>0$,
	 \begin{align*}
	 	|E+B_r(0)| \leq  \frac{C}{\kappa_{E,\eta}}  \cdot \hh(E) \cdot \Big(r+\frac{r^2}{\eta}\Big).
	 \end{align*}
\end{coro}
\begin{proof}
	From Proposition~\ref{prop_ca} and the coarea formula it follows that
	\begin{align*}
		|E+B_r(0)|&= \int_{\R^2} \chi_{[0,r]}(d(x,E)) dx 
		= \int_{\R^2} \chi_{[0,r]}(d(x,E)) | \nabla d(x,E) | dx 
		\\ &= \int_0^r \hh \big(\{x:\ d(x,E)=t\} \big) dt
		\leq  \frac{C}{\kappa_{E,\eta}} \hh(E)  \int_0^r   \Big(1+\frac{t}{\eta}\Big) dt
		 \\ &\leq  \frac{C}{\kappa_{E,\eta}}  \hh(E)  \Big(r+\frac{r^2}{\eta}\Big). \qedhere
	\end{align*}
\end{proof}

\section{The spectrum of erfc Toeplitz operators}\label{sspe}
\subsection{The Bargmann transform}
The \emph{weighted Bargmann-Fock space} \cite{MR157250} is
\begin{align}\label{eq_fock}
\mathcal{F}^2(\C)=\left\{f e^{-|z|^2/2}:\ f \text{ entire and }\|f e^{-|z|^2/2}\|_{L^2(\C,dA)}<\infty\right\}\subseteq L^2(\C,dA).
\end{align} 
For convenience, we consider weighted entire functions with respect to the area measure, though it is more common to consider analytic functions with respect to the Gaussian measure. The reproducing kernel of $\mathcal{F}^2(\C)$ is the Ginibre kernel \eqref{egin}, and thus provides an integral kernel for the orthogonal projection $P: L^2(\C,dA) \to \mathcal{F}^2(\C)$.

The \emph{weighted Bargmann transform} is the unitary operator $B:L^2(\R)\to \mathcal{F}^2(\C)\subseteq L^2(\C,dA)$ given by
\begin{align*}
	B\varphi(z)
	= \Big(\frac{2}{\pi}\Big)^{1/4}\int_{\mathbb{R}}\varphi(t)e^{2tz-t^2-z^2/2-|z|^2/2}dt,
	\quad z \in \C, \varphi\in L^2(\R),
\end{align*}
or, with notation $z=z_1+iz_2\in \C$, 
\begin{align}\label{eber}
	B\varphi(z)
	= \Big(\frac{2}{\pi}\Big)^{1/4}\int_{\mathbb{R}}\varphi(t)e^{-(z_1-t)^2}e^{2itz_2-iz_1z_2}dt.
\end{align}
The adjoint transformation $B^*:  L^2(\C,dA)\to L^2(\R)$ is given by
\begin{align*}
	B^*f(t)
	= \Big(\frac{2}{\pi}\Big)^{1/4}\int_{\C}f(z)e^{2t\overline{z}-t^2-\overline{z}^2/2-|z|^2/2}dA(z),
	\quad t \in \R, f\in L^2(\C,dA).
\end{align*}
The relevance of the Bargmann transform for us is that the integral kernel of 
$B B^*: L^2(\C,dA) \to L^2(\C,dA)$ is precisely the Ginibre kernel \eqref{egin}. In other words, 
\begin{align*}
P:=P_{\infty}=B B^*.
\end{align*}
We now extend these observations to the translated erfc-kernel $F_l$ \eqref{eq_fl}; see \cite{AmKaMa} for closely related facts. The operator with integral kernel $F_l$ will be denoted $P_{l}:L^2(\C,dA) \to L^2(\C,dA) $.

\begin{lemma}\label{lemcd}
	For every $l\in(-\infty,\infty]$, we have that $P_l=B M_{(-\infty,l)} B^*$.
\end{lemma}
\begin{proof}
	The discussion above gives the case $l=\infty$. Let $l\in\R$ and note that
	\begin{align*}
		F_l(z,w) &= e^{z\overline{w}-|z|^2/2-|w|^2/2} \frac{1}{\sqrt{2\pi}} \int_{-\infty}^0 e^{-(z + \overline{w}-2l-t)^2/2} dt
		\\ & =  \sqrt{\frac{2}{\pi }} e^{z\overline{w}-|z|^2/2-|w|^2/2}  \int_{-\infty}^{l} e^{-(z + \overline{w}-2s)^2/2} ds
		\\ & = \sqrt{\frac{2}{\pi }}  \int_{-\infty}^{l} e^{2sz-s^2-|z|^2/2-z^2/2} e^{2s\overline{w}-s^2-|w|^2/2-\overline{w}^2/2} ds.
	\end{align*}
	With this at hand, for $f\in L^2(\C,dA)$ and $z\in\C$,
	\begin{align*}
	    P_{l} f (z) &=   \int_\C f(w) F_l(z,w) dA(w)
	    \\ &=   \int_\C f(w)  \sqrt{\frac{2}{\pi }}  \int_{-\infty}^{l}e^{2sz-s^2-|z|^2/2-z^2/2} e^{2s\overline{w}-s^2-|w|^2/2-\overline{w}^2/2} ds dA(w)
	    \\ &=   \sqrt{\frac{2}{\pi }} \int_{-\infty}^{l} \int_\C f(w)   e^{2s\overline{w}-s^2-|w|^2/2-\overline{w}^2/2}  dA(w) e^{2sz-s^2-|z|^2/2-z^2/2}  ds
	    \\ &=  \Big(\frac{2}{\pi}\Big)^{1/4}  \int_{-\infty}^{l}  B^*f(s) e^{2sz-s^2-|z|^2/2-z^2/2} ds = B M_{(-\infty,l)} B^* f(z). \qedhere
	\end{align*}
\end{proof}
We also note the following covariance property, that will allow us to eliminate the parameter $l$.
\begin{lemma}\label{ltrans}
For every $l\in \R$, let $U: L^2(\C)\to L^2(\C)$ be the unitary operator
\[Uf(z_1+iz_2)=e^{-ilz_2}f(z_1+l+iz_2).\]
Then
\[P_l  M_{\Omega} P_l= U^{-1} P_0  M_{\Omega-l} P_0 U.\]
In particular, both operators share the same spectrum.
\end{lemma}
\begin{proof}
For $a\in\R$, define the translation operator $T_a:L^2(\R)\to L^2(\R)$ given by
\[T_a\varphi (t)=\varphi(t+a), \qquad \varphi\in L^2(\R).\]
From \eqref{eber}, for $\varphi\in L^2(\R)$,
\begin{align*}
	UB\varphi(z)&=e^{-ilz_2} \Big(\frac{2}{\pi}\Big)^{1/4}\int_{\mathbb{R}}\varphi(t)e^{-(z_1+l-t)^2}e^{2itz_2-i(z_1+l)z_2}dt
	\\ &=\Big(\frac{2}{\pi}\Big)^{1/4}\int_{\mathbb{R}}\varphi(t)e^{-(z_1+l-t)^2}e^{2i(t-l)z_2-iz_1z_2}dt
	\\ &=\Big(\frac{2}{\pi}\Big)^{1/4}\int_{\mathbb{R}}\varphi(t+l)e^{-(z_1-t)^2}e^{2itz_2-iz_1z_2}dt=B T_{l}\varphi(z).
\end{align*}
So we get $UB=BT_l$ and taking adjoints, $B^* U^{-1}=T_{-l}B^*.$ By Lemma~\ref{lemcd},
\begin{align*}
	UP_l U^{-1}=UB M_{(-\infty,l)} B^* U^{-1}=BT_l M_{(-\infty,l)} T_{-l}B^*
	= B M_{(-\infty,0)} B^*=P_0,
\end{align*}
where the second-to-last equality follows from the fact that,
\[T_l M_{(-\infty,l)} T_{-l}\varphi(t)=T_l(\chi_{(-\infty,l)}(t)\varphi(t-l))
= \chi_{(-\infty,l)}(t+l)\varphi(t) =M_{(-\infty,0)}\varphi(t).\]
Therefore,
\begin{gather*}
	UP_l M_{\Omega} P_l U^{-1}= P_0  U M_{\Omega} U^{-1} P_0= P_0  M_{\Omega-l}  P_0,
\end{gather*}
where the last equality follows from the fact that,
\[U M_{\Omega} U^{-1}f(z)=U(\chi_{\Omega}(z)e^{ilz_2}f(z-l))
= \chi_{\Omega}(z+l)f(z) =M_{\Omega-l}f(z). \qedhere\]
\end{proof}

\subsection{Comparing erfc and Bargmann concentration operators}

With Lemma \ref{ltrans} at hand, we focus on studying $P_0  M_{\Omega} P_0$. For a positive compact operator $T$ we arrange its eigenvalues in decreasing order and denote them by $\lambda_n(T)$.
We start with the following observation.
\begin{lemma}\label{clam}
	For every compact domain $\Omega\subseteq \C$ and $n\in\N$,
	\begin{align*}
		0\le \lambda_n(P_{0} M_\Omega P_{0})\le \lambda_n(P M_{\Omega} P).
	\end{align*}
	In particular, for every $0<\alpha<1$,
	\begin{equation*}
		\#\{n\in\N: \ \lambda_n(P_{0} M_\Omega P_{0})>\alpha\}
		\le \#\Big\{n\in\N: \ \lambda_n(P M_{\Omega} P)>\alpha\Big\}.
	\end{equation*}
\end{lemma}
\begin{proof}
Since $B^*=I B^*=B^*B B^*=B^* P$ and $B= B I=B B^* B= P B$, the operator
\[P_0  M_{\Omega} P_0= P_0 P M_{\Omega} P P_0\]
is a compression of the concentration operator $P M_{\Omega} P$, and thus has smaller singular values.
\end{proof}
Next we look into estimating $\lambda_n(P_{0} M_\Omega P_{0})$ \emph{from below} in terms of $P M_{\Omega} P$, by quantifying how much compression by $P_0$ affects the spectrum.
Inspecting \eqref{eber}, we see that if a function $\varphi\in L^2(\R)$ is supported in $(-\infty,0)$, then its Bargmann transform $B \varphi$ should be concentrated near the half-plane $\{z\in \C: \ z_1\le 0\}$. Similarly, $P_0$ should not significantly alter functions supported well inside $\{z\in \C: \ z_1\le 0\}$. We now formalize this intuition. Recall that $F$ denotes the complementary error function \eqref{eerf}.

\begin{lemma}\label{llow}
	Let $a\ge 0$ and $\Omega\subseteq \C$ be a compact domain such that $|\Omega\cap \{z\in \C: \ z_1> a\}|=0$. For every $n\in \N$,
	\begin{align}
		\label{eeig}
		\lambda_n(P_{0} M_\Omega P_{0})^{1/2}\ge 
		\lambda_n(P M_\Omega P)^{1/2}- \Big(\tfrac{F(-2a)}{F(2a)}\Big)^{1/2} (1-\lambda_n(P M_\Omega P))^{1/2}.
	\end{align}
	As a consequence, for $0<\alpha<1$,
	\begin{equation}\label{eq_b}
		\#\{n\in\N: \ \lambda_n(P_{0} M_\Omega P_{0})>1-\alpha\}
		\ge \#\Big\{n\in\N: \ \lambda_n(P M_\Omega P)>1-\tfrac{F(2a)}{4F(-2a)}\alpha^2\Big\}.
	\end{equation}
\end{lemma}
\begin{proof}
	We first note that $P M_\Omega P$ regarded as an operator $\mathcal{F}^2(\C) \to \mathcal{F}^2(\C)$ is unitary conjugate to $B^* M_\Omega B$ via the Bargmann transform $B$. In particular, both operators share the same spectrum. On the other hand, considering $P M_\Omega P$ as an operator in $L^2(\C,dA)$ only adds zero eigenvalues. Thus, $\lambda_n(P M_\Omega P)=\lambda_n( B^* M_\Omega B)$ for every $n\in\N$. Arguing similarly with $M_{(-\infty,0)} B^* M_\Omega B M_{(-\infty,0)}$ and $P_{0} M_\Omega P_{0}$, we conclude that we can work with $B^* M_\Omega B$ and $M_{(-\infty,0)} B^* M_\Omega B M_{(-\infty,0)}$ instead of $P M_\Omega P$ and $P_{0} M_\Omega P_{0}$ respectively.
	
	Fix $n\in \N$, write $\lambda_k=\lambda_k(B^* M_\Omega B)$, and consider $\{\psi_k\}_k$ an associated orthonormal basis of eigenvectors. 
	By the Fisher-Courant formula for eigenvalues,
	\begin{align}\label{ecou}
	\begin{aligned}
		\lambda_n(M_{(-\infty,0)} B^* M_\Omega B M_{(-\infty,0)})
		&= \max_{U,\dim U\le n } \min_{\psi\in U,\|\psi\|=1} \langle M_{(-\infty,0)} B^* M_\Omega B M_{(-\infty,0)} \psi, \psi\rangle
		\\ &\ge \min_{\psi \in \langle \psi_k \rangle_{k=1}^n,\|\psi\|=1}  \langle M_{(-\infty,0)} B^* M_\Omega B M_{(-\infty,0)} \psi, \psi\rangle.
	\end{aligned}
	\end{align}
	Fix $\psi\in \langle \psi_k \rangle_{k=1}^n$ such that $\|\psi\|=1$. We have,
	\[\psi=\sum_{k=1}^n a_k \psi_k, \quad \text{with} \quad \sum_{k=1}^n |a_k|^2=1.\]	
	Note that
	\begin{align}\label{eqsqr}
		\langle M_{(-\infty,0)} B^* M_\Omega B M_{(-\infty,0)} \psi, \psi\rangle^{1/2} &= \Big(\int_\Omega |B(\chi_{(-\infty,0)}\psi)|^2 \Big)^{1/2} 
		\\ &\ge \Big(\int_\Omega |B(\psi)|^2 \Big)^{1/2}-\Big(\int_\Omega |B(\chi_{(0,\infty)}\psi)|^2 \Big)^{1/2}. \notag
	\end{align}
	We study both integrals separately. First,
	\begin{align}\label{eqlam}
		\lambda := \int_\Omega |B \psi|^2 =\langle B^* M_\Omega B \psi, \psi \rangle =  \sum_{k=1}^n |a_k|^2 \lambda_k \ge \lambda_n.
	\end{align}
	Secondly, using \eqref{eber} and Plancherel’s identity,
	\begin{align}\label{eqcut}
		\int_\Omega |B(\chi_{(0,\infty)}\psi)|^2  dA
		&\le  \pi^{-1}\int_{-\infty}^{a} \int_{\R} |B(\chi_{(0,\infty)}\psi)(z)|^2  dz_2 dz_1
		\\ & = \sqrt{2}\pi^{-3/2}\int_{-\infty}^{a} \int_{\R} \Big|
		\int_{\mathbb{R}}\chi_{(0,\infty)}(t)\psi(t)e^{-(z_1-t)^2}e^{2i(t-z_1)z_2}dt\Big|^2  dz_2 dz_1 \notag
		\\ & =\sqrt{\frac{2}{\pi}} \int_{-\infty}^{a} \int_{\R} \Big|
		\int_{\mathbb{R}}\chi_{(0,\infty)}(t)\psi(t)e^{-(z_1-t)^2}e^{-2\pi itx}dt\Big|^2  dx dz_1 \notag
		\\ & =\sqrt{\frac{2}{\pi}} \int_{-\infty}^{a} \int_{0}^\infty | \psi(t)e^{-(z_1-t)^2}|^2 dt dz_1 \notag
		\\ &= \sqrt{\frac{2}{\pi}}\int_{0}^\infty | \psi(t)|^2 \int_{-\infty}^0   e^{-2(s-t+{a})^2}  ds  dt \notag
		\\ &= \int_{0}^\infty | \psi(t)|^2 F(2t-2a)  dt \le F(-2a) \int_{0}^\infty |\psi(t)|^2  dt.  \notag 
	\end{align}
	Similarly,
	\begin{align*}
		1-\lambda= \int_{\Omega^c} |B \psi|^2 dA &\ge \pi^{-1} \int_{a}^\infty \int_\R |B \psi(z)|^2 dz_2 dz_1
		\\ &=  \sqrt{\frac{2}{\pi}} \int_{a}^\infty \int_{\R} \Big|
		\int_{\mathbb{R}}\psi(t)e^{-(z_1-t)^2}e^{-2\pi itx}dt\Big|^2  dx dz_1 
		\\ &= \sqrt{\frac{2}{\pi}} \int_{a}^\infty \int_\R |\psi(t)|^2 e^{-2(z_1-t)^2} dt dz_1 
		\\ &= \int_\R |\psi(t)|^2 F(-2t+2a) dt  \ge F(2a) \int_{0}^\infty |\psi(t)|^2  dt.
	\end{align*}
	Combining this with \eqref{eqsqr}, \eqref{eqlam} and \eqref{eqcut} we obtain
	\begin{align}\label{eq_a}
	\begin{aligned}
		\langle M_{(-\infty,0)} B^* M_\Omega B M_{(-\infty,0)} \psi, \psi\rangle^{1/2} &\ge \lambda^{1/2}- \Big(\frac{F(-2a)(1-\lambda)}{F(2a)}\Big)^{1/2}
		\\ &\ge \lambda_n^{1/2}- \Big(\frac{F(-2a)(1-\lambda_n)}{F(2a)}\Big)^{1/2}, 
	\end{aligned}
	\end{align}
	where in the last step we used that the functions $\sqrt{x}$ and $-\sqrt{(1-x)}$ are increasing in $(0,1)$. Combining \eqref{ecou} and \eqref{eq_a} yields \eqref{eeig}. 
	
	In addition, if 
	\[\lambda_n>1-\frac{F(2a)}{4F(-2a)}\alpha^2, \quad \text{with } 0<\alpha<1,\]
	then, since $F(2a)\le F(-2a)$,
	\begin{align*}
		\lambda_n(M_{(-\infty,0)} B^* M_\Omega B M_{(-\infty,0)})^{1/2}&\ge 
		\Big(1-\frac{\alpha^2}{4}\Big)^{1/2} - \frac{\alpha}{2}.
	\end{align*}
	Since the right-hand side of the last equation is positive,
	\begin{align*}
		\lambda_n(M_{(-\infty,0)} B^* M_\Omega B M_{(-\infty,0)})&\ge 
		1-\frac{\alpha^2}{4} -2\frac{\alpha}{2}\Big(1-\frac{\alpha^2}{4}\Big)^{1/2} + \frac{\alpha^2}{4}
		> 1-\alpha,
	\end{align*}
which proves \eqref{eq_b}.
\end{proof}

\subsection{Spectral deviation for Bargmann concentration operators}
We now give the following spectral deviation estimates for concentration operators on the Bargmann-Fock space, taken from \cite{MR}.

\begin{theorem}\label{teig}
	Let $\Omega\subseteq\C$ be a compact set with regular boundary at scale $\eta >0$. For $\alpha\in(0,1)$ set $\tau=\max\left\{\tfrac{1}{\alpha}, \tfrac{1}{1-\alpha}\right\}$. Then, there is a universal constant $C$ such that,
\begin{align}\label{eq_c}
	\Big| \#\{k\in\N: \lambda_k(P M_{\Omega}P) > \alpha\} - \frac{|\Omega|}{\pi} \Big|
	\leq C 
	\hh(\partial \Omega)   \left(\sqrt{\log \tau}+\frac{\log \tau}{\eta }\right)
	\frac{	\log\big[1+\log(\tau)\big]}{\cu }.
\end{align}
\end{theorem}
\begin{proof}
The results in \cite{MR} concern concentration operators on certain Hilbert spaces, one of which can be identified with the weighted Bargmann-Fock space \eqref{eq_fock} after a change of variables. More precisely, the \emph{short-time Fourier transform} of a function $f \in L^2(\mathbb{R})$ with respect to the window function $g(t)=e^{-t^2}(2/\pi)^{1/4}$ is a function $V_g f \in L^2(\mathbb{R}^2)$ related to the weighted Bargmann transform by 
\[Bf(z)= e^{-iz_1z_2} V_gf(z_1,-z_2/\pi), \quad z\in \C, f \in L^2(\R).\]
As explained in \cite[Remark 1.2]{MR}, the Gaussian window function $g$ belongs to Gelfand-Shilov space with parameter $\beta = 1/2$. Therefore, Theorem~1.1 in \cite{MR} (with $d = 1$ and Gelfand-Shilov type condition with parameter $\beta = 1/2$) provides a spectral estimate analogous to \eqref{eq_c} for concentration operators on the range space $V_g L^2(\mathbb{R})$. The change of variables $(x,y) \to x-i\pi y$ then yields \eqref{eq_c}.
\end{proof}

\subsection{Spectral deviation for erfc concentration operators}
We finally combine all parts of our previous analysis.

\begin{theorem}\label{pfad}
	Let $\Omega\subseteq \C$ be a compact domain with regular boundary at scale $\eta>0$ and $l\in (-\infty,\infty]$. There is a universal constant $C$ such that the following statements hold.
	\begin{enumerate}[label={(\roman*)}]
		\item\label{pfadi} For every $0<\alpha <1/2$,
		\[\#\{k\in\N: \lambda_k(P_{l} M_{\Omega}P_{l}) > \alpha\}\le \frac{|\Omega|}{\pi}+\frac{C}{\cu} \hh(\partial\Omega) \Big(\sqrt{\log \alpha^{-1}}+\frac{\log \alpha^{-1}}{\eta}\Big)\log(1+\log \alpha^{-1}).\]
		
		\item\label{pfadii} Let $a\in\R$ be such that
		\[|\Omega \cap \{\re(z)> a+l\}|=0.\]
		Then for $0<\alpha<1/2$,
		\[\#\{k\in\N: \lambda_k(P_{l} M_{\Omega}P_{l}) > 1-\alpha\}\ge \frac{|\Omega|}{\pi}-\frac{C h(a)}{\cu} \hh(\partial\Omega) \Big(\sqrt{\log \alpha^{-1}}+\frac{\log \alpha^{-1}}{\eta}\Big)\log(1+\log \alpha^{-1}),\]
		where
		\[h(x)=\begin{cases}
			1 & \text{if } x\le 2
			\\	x^2\log x & \text{if } x> 2
		\end{cases}.\] 
	\end{enumerate}
\end{theorem}

\begin{proof}
If $l=\infty$ we can apply Theorem~\ref{teig} directly, so we assume $l\in\R$.
By Lemmas~\ref{ltrans} and \ref{clam} we have
\begin{align*}
	\#\{k\in\N: \lambda_k(P_{l} M_{\Omega}P_{l}) > \alpha\}&=\#\{k\in\N: \lambda_k(P_{0} M_{\Omega-l}P_{0}) > \alpha\}
	\\ &\le \#\{k\in\N: \lambda_k(P M_{\Omega-l}P) > \alpha\}.
\end{align*}
Part \ref{pfadi} now follows from Theorem~\ref{teig}.

Regarding part \ref{pfadii}, notice that the case $a=0$ implies the case $a<0$, so we assume that $a\ge 0$ without loss of generality. The set $\Omega-l$ satisfies the hypothesis of Lemma~\ref{llow} since
\[|(\Omega-l) \cap \{\re(z)> a\}|=|\Omega \cap \{\re(z)> a+l\}|=0.\]
Hence, by Lemmas~\ref{ltrans} and \ref{llow},
\begin{align*}
	\#\{n\in\N: \ \lambda_n(P_{l} M_\Omega P_{l})>1-\alpha\}
	&=\#\{n\in\N: \ \lambda_n(P_{0} M_{\Omega-l} P_{0})>1-\alpha\}
	\\ &\ge \#\Big\{n\in\N: \ \lambda_n(P M_{\Omega-l} P)>1-\frac{F(2a)}{4F(-2a)}\alpha^2\Big\}.
\end{align*}
Notice that since $a\ge 0$, $F(2a)\le F(-2a)$ and therefore,
\[\frac{F(2a)}{4F(-2a)}\alpha^2\le \frac{1}{2}.\]
So, applying Theorem~\ref{teig},
\begin{multline*}
	\#\{n\in\N: \ \lambda_n(P_{l} M_\Omega P_{l})>1-\alpha\}
 \ge \frac{|\Omega|}{\pi}-\frac{C}{\cu} \hh(\partial\Omega) \Big(\sqrt{\log \tau}+\frac{\log \tau}{\eta}\Big)\log(1+\log \tau),
\end{multline*}
where
\[\tau= \frac{4 F(-2a)}{F(2a)\alpha^2}.\]
Since $\alpha< 1/2$ we get
\begin{multline*}
	\#\{n\in\N: \ \lambda_n(P_{l} M_\Omega P_{l})>1-\alpha\}
	\ge \frac{|\Omega|}{\pi}-C  \frac{g(a)}{\cu} \hh(\partial\Omega)  \Big(\sqrt{\log \alpha^{-1}}+\frac{\log \alpha^{-1}}{\eta}\Big)\log(1+\log \alpha^{-1}),
\end{multline*}
where
\[g(x)=\max\big\{1, \log (F(2x)^{-1})\log [1+\log (F(2x)^{-1})]\big\}.\]
The estimate involving the function $h$ now follows from a straightforward computation using that, for $x \geq 0$,
\begin{align*}
	F(2x)&=\frac{1}{\sqrt{2\pi}} \int_{-\infty}^0 e^{-(2x-t)^2/2} dt
	=\sqrt{\frac{2}{\pi}} \int_{x}^{\infty} e^{-2s^2} ds
	\ge\sqrt{\frac{2}{\pi}} \int_{x}^{x+1} e^{-2s^2} ds
	\\ &\ge \sqrt{\frac{2}{\pi}} e^{-2(x+1)^2} \ge \sqrt{\frac{2}{\pi}} e^{-4(x^2+1)}.\qedhere 
\end{align*}
\end{proof}

\section{Discrepancy at the boundary}\label{sbou}
\subsection{Landau's method}
The main goal in this section is to prove Theorem~\ref{cdis}.
We assume throughout that $Q$ satisfies the conditions \ref{c1} to \ref{c7}, and let $S$ denote the associated droplet.

We first collect some consequences of \cite{amro20}. Throughout this section we use the following notation: for a compact domain $\Omega$, $s>0$, and a sequence of zooming points $\{p_n\}_{n=1}^\infty \subset \mathbb{C}$ we write
\begin{align*}
	\Omega_n &=p_n+\tfrac{1}{\sqrt{n}}\Omega;
	\\\Omega_n^+ &=p_n+\tfrac{1}{\sqrt{n}}\{z\in\C: \ \dd(z,\Omega)<s\};
	\\\Omega_n^- &=p_n+\tfrac{1}{\sqrt{n}}\{z\in\Omega: \ \dd(z,\partial\Omega)>s\}.
\end{align*}
The number of random samples of the Coulomb gas in these domains are denoted
\begin{align}\label{eq_aa}
	N_n&=\#[\Omega_n], \qquad N_n^\pm=\#[\Omega_n^\pm].
\end{align}
In the low temperature regime \eqref{eq_lt}, these quantities can be estimated by means of the following lemma, which is the starting point of our proofs.
\begin{lemma}\label{lemma_quote}
Under the assumptions of Theorem~\ref{cdis}, let $\delta\in (0,1)$ 
be a failure probability and $\gamma\in(0,1)$. Then there exist positive constants $s,C,M$ depending only on $c$ and $Q$, and $n_0=n_0(c,\delta,\gamma)$ such that,
with probability at least $1-\delta$, the following hold simultaneously for all $n \geq n_0$.
\begin{itemize}
\item[(i)] The points $\{z_j\}_j$ are $s/\sqrt{n}$-separated,
\begin{align*}
\sqrt{n} \min_{1 \leq j \not= k \leq n} |z_j-z_k| \geq s,
\end{align*}
and contained in $S_M=S+B_{M/\sqrt{n}}(0)$, where $S$ is the droplet.

\item[(ii)] The counting functions \eqref{eq_aa} satisfy
\begin{align}\label{ecard0}
	N_n^+-N_n^-\lesssim \frac{1+1/\eta}{\cu} \cdot \hh(\partial\Omega),
\end{align}
where the implied constant depends on $c$ and $Q$.
\item[(iii)] Let $\rho>0$. Then the spectrum of the \emph{concentration operator} $T_{\rho,n}:L^2(\C,dA)\to L^2(\C,dA)$,
\begin{align}\label{eq_top}
T_{\rho,n} f=P_{\ww_{\lfloor \rho n\rfloor }}(\chi_{\Omega_n\cap S_{M+s}} P_{\ww_{\lfloor \rho n\rfloor}}f),
\end{align}
satisfies, for $\gamma \in (0,1)$,
\begin{align}\label{ecard}
\begin{aligned}
	N_n^+&\ge\#\{k\in\N: \lambda_k(T_{1-\gamma,n}) > 1-\gamma^2/C\}, 
	\\ N_n^-&\le\#\{k\in\N: \lambda_k(T_{1+\gamma,n}) > \gamma^2/C\}.
\end{aligned}
\end{align}
\end{itemize}  
\end{lemma}
\begin{proof}
The separation claim in (i) is proved in \cite{amro20}, while the containment follows from the more general ``localization theorem''	in \cite{Am}; for a definite citation, Theorem 5.1 in \cite{amro20} contains both statements exactly as quoted.

Since the points $\{z_j\}_j$ are $s/\sqrt{n}$-separated, we can invoke
Corollary~\ref{coar} and estimate
\[N_n^+-N_n^-\lesssim {s}^{-2} |\{z\in\C: \ \dd(z,\partial\Omega)\le s\}|\lesssim \frac{\hh(\partial\Omega)}{\cu} {s}^{-2} (s+s^2/\eta).\]
Since $s$ depends only on $c$ and $Q$, this proves (ii).

Finally, the estimates in (iii) are contained in the proof of \cite[Theorem 1.2]{amro20}; specifically, see Equations (6.8) and (6.13) in \cite{amro20}. 	
\end{proof}

\subsection{Convergence of spectra}
The following lemma helps describe the limiting spectral profile of concentration operators $T_{\rho,n}$ on spaces of weighted polynomials under the scaling discussed in Section \ref{sec_sc}.

\begin{lemma}\label{llim}
	Let $p_n\to p\in  S$ with parameters $\theta\in[0,2\pi)$ and $l\in(-\infty,\infty]$ as in Definition~\ref{dlim}, $\rho\in (0,2)$ and write $m=\sqrt{\rho \Delta Q(p)}$. Let $\Omega\subset\C$ be compact, $s,M>0$, define $T_{\rho,n}$ by \eqref{eq_top} and set 
	\begin{align}\label{eq_limset}
	\widetilde{\Omega}=m (e^{-i\theta}\Omega \cap \{z: \re(z)\le M+s+l\}) \quad \text{and} \quad T_\rho=P_{{m l }} M_{\widetilde{\Omega}}P_{{ ml}}.
	\end{align} 
	
	Then for $\alpha\in (0,1)\smallsetminus \{\lambda_k(T_\rho)\}_k$, we have,
	\begin{align*}
		\lim_{n\to\infty}\#\{k\in\N: \lambda_k(T_{\rho,n}) > \alpha\}
		=\#\{k\in\N: \lambda_k(T_\rho) > \alpha\}.
	\end{align*}
\end{lemma}
\begin{proof}
Without loss of generality, we assume that for all $n \in \mathbb{N}$, $p_n$ belongs to the neighborhood of $S$ referred to in the assumptions on the potential $Q$.
	
Fix $g_n$ as in Lemma~\ref{lklim}. We write the zooming map $\Gamma:\C\to\C$ for $\rho n$,
\[\Gamma(z)= \sqrt{\lfloor \rho n\rfloor \Delta Q(p_n)} (z-p_n)e^{-i\theta_n},\]
and define its associated pull-back $\Gamma':L^2(\C,dA) \to L^2(\C,dA)$, $\Gamma' f= f\circ \Gamma$. Consider the operators
\begin{align*}
	\til_{n}= M_{g_{\lfloor \rho n\rfloor}}^{-1} \Gamma'^{-1} T_{\rho,n} \Gamma'  M_{g_{\lfloor \rho n\rfloor}}, \qquad \text{and} \qquad  R_n=\til_{n}-T_\rho.
\end{align*}
	
By the Courant-Fisher characterization of eigenvalues of self-adjoint operators, 
\[|\lambda_k(T_\rho+R_n)- \lambda_k(T_\rho)|\le \|R_n\|.\] 
Therefore, it suffices to show that $\|R_n\|\to 0$, since in that case, for every $\alpha\in (0,1)\smallsetminus \{\lambda_k(T)\}_k$,
\[\#\{k\in \N: \ \lambda_k(T_{\rho,n})>\alpha\}= \#\{k\in \N: \ \lambda_k(T_\rho+R_n)>\alpha\}\to \#\{k\in \N: \ \lambda_k(T_\rho)>\alpha\}.\]

Let us show that $\|R_n\|\to 0$.
First, a direct computation shows that $\til_n$ has integral kernel 
\begin{align}\label{etou}
	K_{\til_{n}}(z,w)&=\int_{\Gamma(\Omega_n\cap S_{M+s})}  \mathcal{K}_{n}(z,u) \mathcal{K}_{n}(u,w) \,dA(u),
\end{align}
where
\[\mathcal{K}_{n}(z,w)=g_{\lfloor \rho n\rfloor}(z)\overline{g_{\lfloor \rho n\rfloor}(w)} \frac{K_{\lfloor \rho n\rfloor}(\Gamma^{-1}(z),\Gamma^{-1}(w))}{{\lfloor \rho n\rfloor} \Delta Q(p_n)}.\]
In addition,
\[\Gamma(\Omega_n\cap S_{M+s})= \sqrt{\tfrac{\lfloor \rho n\rfloor}{n} \Delta Q(p_n)} e^{-i\theta_n}(\Omega\cap \sqrt{n}(S_{M+s}-p_n))\]
is uniformly bounded. Due to the assumed regularity of $\partial S$,
\[\mathbbm{1}_{\Gamma(\Omega_n\cap S_{M+s})}\xrightarrow{L^2} \mathbbm{1}_{\widetilde{\Omega}};\]
see \cite[Lemma~6.2]{amro20} for details.

On the other hand, Lemma~\ref{lklim} provides a uniform limit for the integrand on the right-hand side of \eqref{etou}. Thus, we have the pointwise limit:
\begin{align}\label{eq_l}
K_{\til_{n}}(z,w)\xrightarrow{n\to\infty} 
K_{T_\rho}(z,w):=
\int_{\widetilde{\Omega}} F_{ml}(z,u)F_{ml}(u,w) \,dA(u), \qquad z,w\in \C,
\end{align}
which is the integral kernel of $T_\rho$.
To fully justify the use of Lemma~\ref{lklim}, we indicate how to sort out the technicalities concerning integer parts. Consider first the set of numbers $I := \{ \lfloor \rho n\rfloor: n \in \mathbb{N} \}$ enumerated as $\{n_k: k \geq 1\}$. For each $k \in \N$, we select one $n \in \mathbb{N}$ such that $n_k=\lfloor \rho n\rfloor$ and define $\tilde{p}_{n_k} := p_n$. Using Lemma~\ref{lext}, we extend this sequence to a full sequence $\{\tilde{p}_n : n \geq 1\}$ whose microscopic position relative to the boundary is $\sqrt{\rho}\cdot l$. We then invoke Lemma~\ref{lklim} with $\tilde{p}_n$ and pass to the subsequence indexed by $n_k$. If $\rho\ge 1$ we are done, since $(\tilde{p}_{n_k})_{k\in\N}=(p_n)_{n\in\N}$. If $\rho< 1$ however, $(\tilde{p}_{n_k})_{k\in\N}$ covers only some of the $p_n$. To obtain \eqref{eq_l} for the full sequence we note that for each $n_k \in I$, the number of possible choices satisfies $\# \{n \in \mathbb{N}: \lfloor \rho n\rfloor=n_k\}\le \lceil \rho^{-1} \rceil$, so we can cover them all repeating the previous argument $\lceil \rho^{-1} \rceil$ times.

Similarly,
\begin{align}\label{eq_ll}
\begin{aligned}
	\tr ( \til_{n}^2)&=\int_{\Gamma(\Omega_n\cap S_{M+s})} \int_{\Gamma(\Omega_n\cap S_{M+s})} |\mathcal{K}_{n}(z,w)|^2  \,dA(z) dA(w)
\\	&\xrightarrow{n\to\infty} \int_{\widetilde{\Omega}} \int_{\widetilde{\Omega}} |F_{ml}(z,w)|^2  \,dA(z) dA(w) = \tr(T_\rho^2).
\end{aligned}
\end{align}
The pointwise convergence of the kernels \eqref{eq_l} together with the convergence of the $L^2$-norms \eqref{eq_ll} ensures,
by the Brezis--Lieb lemma, that the kernels $K_{\til_{n}}$ converge in $L^2(\C\times\C, dA\times dA)$ to $K_{T_\rho}$, the integral kernel of $T_\rho$. Finally, by Cauchy-Schwartz,
\[\|R_n f\|_{2}=\|(\til_{n}-T_\rho) f\|_{2}\le \|K_{\til_{n}}- K_{T_\rho}\|_{L^2(\C\times\C, dA\times dA)}\|f\|_{2}, \qquad  f\in L^2(\C, dA).\]
Thus, we conclude that $\|R_n\|\to 0$.
\end{proof}

We aim to apply our spectral bounds to the Toeplitz operator in \eqref{eq_limset}; as a first step, we provide the following estimate for the area of the corresponding concentration domain.

\begin{lemma}\label{lmeas}
	With assumptions and notation of Lemma \ref{llim}, the following holds:
	\[	\left||\widetilde{\Omega}|-{ \Delta Q(p)}|\Omega \cap \{z: \re(z e^{-i\theta}) \leq l\}|\right|\le |\rho-1|\Delta Q(p) |\Omega |+ 2 \Delta Q(p) (M+s) \hh(\partial \Omega).\]
\end{lemma}
\begin{proof} Note that,
	\begin{align*}
|\widetilde{\Omega}|
		&= m^2(|\Omega \cap \{z: \re(z e^{-i\theta}) \leq l\}|+ |\Omega \cap \{z: l\le \re(z e^{-i\theta}) \leq M+s+l\}|)
		\\  &= \rho  \Delta Q(p) (|\Omega \cap\{z: \re(z e^{-i\theta}) \leq l\}|+ |\Omega \cap \{z: l\le \re(z e^{-i\theta}) \leq M+s+l\}|).
	\end{align*}
In addition, since $\Omega$ is bounded,
\begin{align*}
\{\im(z): z \in \Omega\} = \{\im(z): z \in \partial\Omega\}.
\end{align*}
Since the projection $z \mapsto \im(z)$ is contractive, the (length) measure of the projected set $\{\im(z): z \in \Omega\}$ is thus bounded by $\hh(\partial \Omega)$. Hence,
\begin{align*}
|\Omega \cap \{z: l\le \re(z e^{-i\theta}) \leq M+s+l\}| \le \hh(\partial \Omega) (M+s).
\end{align*}
Combining the previous estimates,
\begin{align*}
	\left||\widetilde{\Omega}|-{ \Delta Q(p)}|\Omega \cap \{z: \re(z e^{-i\theta}) \leq l\}|\right| 
	&\le \begin{multlined}[t]
	|\rho-1|\Delta Q(p)|\Omega \cap \{z: \re(z e^{-i\theta}) \leq l\}|
	\\ + 2  \Delta Q(p)|\Omega \cap \{z: l\le \re(z e^{-i\theta}) \leq M+s+l\}|
\end{multlined}
	\\  & \le |\rho-1|\Delta Q(p)|\Omega |+ 2  \Delta Q(p)(M+s) \hh(\partial \Omega).\qedhere
\end{align*}
\end{proof}

\subsection{Discrepancy}
Theorem~\ref{cdis} will be deduced from the following more general result for compact observation windows with regular boundary.

\begin{theorem}\label{tdis2}
	Assume $Q$ satisfies the conditions \ref{c1} to \ref{c7}, $\beta_n\ge c \log n$ and $\Omega\subseteq \C$ is a  compact domain  with $|\Omega|\ge 2$ and regular boundary at scale $\eta>0$.  Then there exists a deterministic constant $C=C(c,Q)$ such that, almost surely, for every $p_n\to p\in S$ with parameters $\theta\in[0,2\pi)$ and $l\in(-\infty,\infty]$ as in Definition \ref{dlim},
	\begin{align}\label{eq_dddd}
	\begin{aligned}
		&\limsup_{n\to\infty}\left|\#\left[p_n+\Omega/\sqrt{n}\right]- \Delta Q(p) \frac{|\Omega \cap \{z: \re(z e^{-i\theta}) \leq l\}|}{\pi} \right| 
		\\
		&\qquad\qquad\le C \hh(\partial\Omega)\sqrt{\log  |\Omega|}\frac{(1
			+\sqrt{\log  |\Omega|}/\eta)}{\cu}   \log(1+\log |\Omega|).
	\end{aligned}
	\end{align}
\end{theorem}

\begin{proof}
Fix a failure probability $\delta\in (0,1)$, $0<\gamma\le 1/2$ and invoke Lemma~\ref{lemma_quote} to obtain a suitable $n_0\in\N$. By taking $n_0$ large enough, we can assume that $\Delta Q$ is bounded above and below on $S_M$. We use $C,D$ to denote constants depending only on $c$ and $Q$ that may change from line to line. Note that, in particular, $M$ and $s$ are considered constants.

We apply Lemma~\ref{llim} with $\rho=1\pm \gamma \in [\tfrac12,\tfrac32]$, and adopt the notation of the lemma. From \eqref{ecard} we get
\begin{align}\label{ecard2}
\begin{aligned}
	\liminf_{n\to\infty} N_n^+&\ge\#\{k\in\N: \lambda_k(T_{1-\gamma}) > 1-\gamma^2/C\},
	\\ \limsup_{n\to\infty} N_n^-&\le\#\{k\in\N: \lambda_k(T_{1+\gamma}) > \gamma^2/C\},
\end{aligned}
\end{align}
while
\begin{align}\label{eq_ccc}
m = \sqrt{\rho \Delta Q(p)} \asymp 1.
\end{align}
In order to apply Theorem~\ref{pfad}, we first carry out a minor surgery to meet the hypothesis. We shall replace the set $\widetilde{\Omega}$ defined in \eqref{eq_limset} with \[\Sigma=\widetilde{\Omega}\cup m e^{-i\theta}\partial\Omega.\] Since we  added a null area set, the Toeplitz operator $T_\rho$ defined in \eqref{eq_limset} is unaffected by this replacement. On the other hand, by Lemma~\ref{lcut}, $\Sigma$ has regular boundary at scale $ m \eta$ with parameter $\kappa\ge \cu/2$. 

Note that
\begin{align*}
	 \hh(\partial \Sigma) & \le \hh(\partial \widetilde{\Omega})+ m \hh(\partial \Omega) = 2m \hh(\partial \Omega)
	\le 4 \sqrt{ \Delta Q(p)} \hh(\partial \Omega),
\end{align*}
while, by Lemma~\ref{lmeas} with $\rho=1\pm \gamma$, 
\begin{align}\label{earea}
	\left||\Sigma|-{ \Delta Q(p)}|\Omega \cap \{z: \re(z e^{-i\theta}) \leq l\}|\right|&\le \gamma\Delta Q(p) |\Omega |+ 2 \Delta Q(p) (M+s) \hh(\partial \Omega).
\end{align}
Combining this with \eqref{ecard0} and \eqref{ecard2},
\begin{multline*}
	\limsup_{n\to\infty} \big[N_n- \Delta Q(p)\pi^{-1}|\Omega \cap \{z: \re(z e^{-i\theta}) \leq l\}|\big]
		\\ \begin{aligned}[t]
			\le & \limsup_{n\to\infty} N_n^- - \pi^{-1}|\Sigma|  +\limsup_{n\to\infty} (N_n^+-N_n^-)  
	\\ & + \pi^{-1}(|\Sigma|- \Delta Q(p)|\Omega \cap \{z: \re(z e^{-i\theta}) \leq l\}|)
	\\\le & \#\{k\in\N: \lambda_k(T_{1+\gamma}) > \gamma^2/C\} - \pi^{-1}|\Sigma| + D\frac{\hh(\partial\Omega)}{\cu} (1+1/\eta)
	\\ &   +\gamma\Delta Q(p)\pi^{-1} |\Omega |+  \Delta Q(p) (M+s) \hh(\partial \Omega).
		\end{aligned}
\end{multline*}
Since $\gamma\le 1/2$, and assuming without loss of generality that $C\ge 1$, we have $\gamma^2/C<1/2$. Applying Theorem~\ref{pfad} to $T_{1+\gamma}=P_{{m l }} M_{\Sigma}P_{{ ml}}$,
\begin{multline}\label{esup}
	\limsup_{n\to\infty} \big[N_n- \Delta Q(p)  \pi^{-1} |\Omega \cap \{z: \re(z e^{-i\theta}) \leq l\}|\big]
	\\ \begin{aligned}[t]
		\le &  D \sqrt{\Delta Q(p)}\frac{\hh(\partial\Omega)}{\cu}  \Big(\sqrt{\log (C\gamma^{-2})} +\frac{\log (C\gamma^{-2})}{m\eta}\Big)\log(1+\log (C\gamma^{-2}))
		 \\ &  +  D\frac{\hh(\partial\Omega)}{\cu} (1+1/\eta) +\gamma\Delta Q(p) \pi^{-1}|\Omega |+  \Delta Q(p) (M+s) \hh(\partial \Omega)		
		\\ \lesssim & \gamma |\Omega |+  \frac{\hh(\partial\Omega)}{\cu}  \Big(\sqrt{\log \gamma^{-1}}
		+\frac{\log \gamma^{-1}}{\eta}\Big)\log(1+\log \gamma^{-1}).
	\end{aligned}
\end{multline}
(Here we used \eqref{eq_ccc}, and the fact that $\gamma\le 1/2$ to
eliminate the constant $C$ from the logs, absorbing the contribution into the implied multiplicative constant.)

Analogously, from \eqref{ecard0}, \eqref{ecard2} and \eqref{earea},
\begin{multline*}
	\limsup_{n\to\infty} \big[ \Delta Q(p)\pi^{-1}|\Omega \cap \{z: \re(z e^{-i\theta}) \leq l\}|-N_n\big]
	\\ \begin{aligned}[t]
	\le & \pi^{-1}|\Sigma|-\#\{k\in\N: \lambda_k(T_{1-\gamma}) >1- \gamma^2/C\}  + D\frac{\hh(\partial\Omega)}{\cu} (1+1/\eta)
		\\ &   +\gamma\Delta Q(p)\pi^{-1} |\Omega |+  \Delta Q(p) (M+s) \hh(\partial \Omega).
	\end{aligned}
\end{multline*}
Since $\gamma^2/C<1/2$, applying Theorem~\ref{pfad} to $T_{1-\gamma}$ with $a=M+s$, we get
\begin{multline*}
	\limsup_{n\to\infty} \big[ \Delta Q(p)  \pi^{-1} |\Omega \cap \{z: \re(z e^{-i\theta}) \leq l\}|-N_n \big]
	\\ 	 \lesssim \gamma |\Omega |+ \frac{\hh(\partial\Omega)}{\cu}  \Big(\sqrt{\log \gamma^{-1}}
		+\frac{\log \gamma^{-1}}{\eta}\Big)\log(1+\log \gamma^{-1}).
\end{multline*}
Combining this with \eqref{esup}, we conclude that
\begin{align}
	&\limsup_{n\to\infty} \Big|N_n- \Delta Q(p) \pi^{-1}|\Omega \cap \{z: \re(z e^{-i\theta}) \leq l\}|\Big|\notag
	\\ 
	&\qquad\qquad\lesssim \gamma |\Omega |+ \frac{\hh(\partial\Omega)}{\cu}  \Big(\sqrt{\log \gamma^{-1}}
	+\frac{\log \gamma^{-1}}{\eta}\Big)\log(1+\log \gamma^{-1}).\label{eq_ddd} 	
\end{align}
The result now follows by choosing
\[\gamma= \frac{1}{|\Omega|}, \]
which indeed satisfies $\gamma\le 1/2$ because $|\Omega|\ge 2$.
The first term in \eqref{eq_ddd}, $\gamma |\Omega|=1$, can be absorbed into the second one because, by the isoperimetric inequality,
\begin{align}\label{eq_eeee}
\hh(\partial\Omega) \geq 2 \sqrt{\pi |\Omega|} \ge 1, 
\end{align}
while $\cu \leq 2$ --- c.f., \eqref{eq_cubb}.

(It is important to note that throughout the proof the implied constants do not depend on the failure probability $\delta$.)
\end{proof}

\begin{proof}[Proof of Theorem~\ref{cdis}] 
Let us exclude the null probability event where the estimates of Theorem~\ref{tdis2} fail. Similarly, by part (i) of Lemma~\ref{lemma_quote} the event that $\limsup_n \sqrt{n} \max_j d(z_j, S) = \infty$ has null probability, and we exclude it. Fix a sample of the Coulomb gas outside the excluded events, and let
\begin{align*}
M := \limsup_n \sqrt{n} \max_{1\leq j\leq n} d(z_j, S).
\end{align*}
	
For $n\in \N$ and $p\in \C$, denote 
\begin{gather*}
	A_{n,p}= \left|\#\left[p+\Omega/\sqrt{n}\right]- \frac{n}{\pi}\Delta Q(p)  |(p+\Omega/\sqrt{n}) \cap S| \right|, \qquad \text{and,}
	\\ A_n=\sup_{p\in \C}A_{n,p}.
\end{gather*}
Let $\varepsilon_n\to 0$ and choose $p_n\in \C$ so that $A_{n,p_n}\ge A_n-\varepsilon_n$ and take a subsequence $p_{n_k}$ such that
\begin{align*}
	\lim_{k\to\infty}A_{n_k,p_{n_k}}= \limsup_{n\to\infty} A_n.
\end{align*}

\noindent \emph{Case I: (Exterior regime)}.
First, suppose that
$\liminf_k \sqrt{n_k}\, d^\pm(p_{n_k}\partial S)=-\infty$,
where $d^\pm$ is the signed distance from \eqref{edis}. Passing to a further subsequence, we assume that
\[l_{n_k}=\sqrt{n_k}\, d^\pm(p_{n_k}, \partial S) \to -\infty.\]
Then, for sufficiently large $k$,
\begin{align*}
(p_{n_k}+\Omega/\sqrt{n_k}) \cap S_M=\varnothing
\end{align*}
and, consequently,
\[\lim_{k\to \infty} \#\left[p_{n_k}+\Omega/\sqrt{n_k}\right]=0.\]
Hence,
\[\limsup_{n\to\infty} A_n=\lim_{k\to\infty}A_{n_k,p_{n_k}}=0,\]
and the desired estimate holds trivially.

\noindent \emph{Case II: (Interior or boundary regime)}.

We assume that
\[\liminf_{k\to \infty} \sqrt{n_k}\, d^\pm(p_{n_k},\partial S)>-\infty,\]
and, in particular,
\[\lim_{k\to \infty} d(p_{n_k},S)=0.\]

We take a subsequence (still denoted $p_{n_k}$) such that $p_{n_k}\to p_*\in S$, $e^{i\theta_{n_k}}\to e^{i\theta_*}$  and $l_{n_k}\to l_*\in(-\infty,\infty]$, with the parameters $\theta_{n_k}, \theta_*$ as in Definition \ref{dlim}.
Applying Lemma~\ref{lext}, for $n\neq n_k$ we redefine $p_n$ so that $p_n\to p_*$, $e^{i\theta_{n}}\to e^{i\theta_*}$  and $l_{n}\to l_*$; cf. Definition \ref{dlim}. Note that, we still have
\begin{align}\label{elsup2}
	\limsup_{n\to\infty}A_{n,p_n}= \limsup_{n\to\infty} A_n.
\end{align}
 From the continuity of $\Delta Q$ and the regularity of $\partial S$ we get
 \begin{multline*}
	\left|\frac{n}{\pi}\Delta Q(p_{n})  |(p_{n}+\Omega/\sqrt{n}) \cap S|- \Delta Q(p_*)   \frac{|\Omega \cap \{z: \re(z e^{-i\theta_*}) \leq l_*\}|}{\pi} \right|
 	\\ = \left|\Delta Q(p_{n})  \frac{|\Omega \cap \sqrt{n}(S-p_n)|}{\pi}- \Delta Q(p_*)   \frac{|\Omega \cap \{z: \re(z e^{-i\theta_*}) \leq l_*\}|}{\pi} \right|\xrightarrow{n\to\infty} 0.
 \end{multline*}
Thus,
\begin{align*}
	\limsup_{n\to\infty}A_{n,p_n}= \limsup_{n\to\infty}\left|\#\left[p_n+\Omega/\sqrt{n}\right]- \Delta Q(p_*)   \frac{|\Omega \cap \{z: \re(z e^{-i\theta_*}) \leq l_*\}|}{\pi} \right|.
\end{align*}
Now we use Theorem~\ref{tdis2}. By Lemma~\ref{rrb}, as the boundary of $\Omega$ is connected, $\Omega$ has regular boundary at scale $\eta=\hh(\partial\Omega)$ and $1\le \cu\le 2$. Therefore, we can apply Theorem~\ref{tdis2}. We now inspect \eqref{eq_dddd} and argue that the term $(1+\sqrt{\log  |\Omega|}/\eta)/\cu$ can be dropped. To see this, note that $\cu^{-1}\le 1$, as well as the fact that using
\eqref{eq_eeee} and $|\Omega|\ge 2$ we get
\begin{align*}
\frac{\sqrt{\log  |\Omega|}}{\eta}
&=\frac{\sqrt{\log  |\Omega|}}{\hh(\partial\Omega)}
\lesssim 
\sqrt{\frac{\log  |\Omega|}{|\Omega|}}\lesssim 1.
\end{align*}

Thus, we conclude that
\begin{align*}
	\limsup_{n\to\infty}A_{n,p_n}\le C\hh(\partial \Omega)\sqrt{\log  |\Omega|}   \log(1+\log |\Omega|).
\end{align*}
Combining this with \eqref{elsup2} yields
\eqref{eq_a2}.
\end{proof}
	
\section{Discrepancy in the bulk}\label{sbul}

\subsection{Collecting facts}
	
We will use the off-diagonal estimate for $K_n$ from \cite[Corollary~8.2]{AmHeMa}. To avoid introducing more notation, we state the following weaker version tailored to our needs.
	
\begin{lemma}\label{efine}
	Assume $Q$ is $\mathcal{C}^2$ and satisfies Conditions \ref{c1} and \ref{cau}. There exist constants $\varepsilon,C,D,n_0>0$  and $h\in L^1(\C)$ depending on $Q$ such that for $z\in S$ and $w \in \C$ and $n\geq n_0$,
	\begin{align*}
			|K_n(z,w)|\le \begin{cases}
				Cn e^{-\varepsilon \sqrt{n}  |z-w|} & \text{if }  |z-w|\le d(z,S^c)
				\\ C n e^{-\varepsilon \sqrt{n} d(z,S^c)} & \text{if } |w|\le D, \text{ and } |z-w|\ge d(z,S^c)
				\\ \frac{h(w)}{n} & \text{if } |w|\ge D
			\end{cases}.
	\end{align*}
\end{lemma}

\begin{rema}\label{rem_c2}	
We mention that \cite[Corollary~8.2]{AmHeMa} is stated for $Q\ge 1$. As explained in \cite{AmHeMa}, this assumption is made to quantify certain dependencies, that are left implicit in our formulation. Thus the assumption is not necessary in Lemma \ref{efine}. 

Second, for $|w|\ge D$, the following bound holds:
\[|K_n(z,w)|\le C n |w|^{-\gamma n},\]
for certain sufficiently small $\gamma>0$ depending on $Q$. The rather artificial bound $h(w)/n$ is formulated for technical reasons that concern our application. 

In addition, inspection of the proof in \cite{AmHeMa} and the results it relies on shows that $Q$ only needs to be assumed to be $C^2$ near the droplet $S$, which is already implied by Condition \ref{c5}.
\end{rema}

The following lemma is a particular case of \cite[Theorem~4.2]{AmHeMa2} (see also \cite[Theorem~2.1]{MR3010273}) and provides a pointwise comparison between the equilibrium measure and the kernel $K_n$ at the diagonal.
	
\begin{lemma}\label{lpun}
	Assume $Q$ satisfies Conditions \ref{c1} to \ref{cau}. There exist $C,M, n_0>0$ depending on $Q$ such that for every $n\ge n_0$ and every $z\in \C$ with $d(z,S^c)\ge M n^{-1/2}\log n$,
	\[|K_n(z,z)  - n\Delta Q(z)|\le C.\]
\end{lemma}
	
We mention that \cite[Theorem~4.2]{AmHeMa2} is stated for $d(z,S^c)\ge n^{-1/2}\log^2 n$, but the proof only uses $d(z,S^c)\ge M n^{-1/2}\log n$ for some sufficiently large $M>0$, depending only on $Q$. Moreover, \cite[Theorem~2.1]{MR3010273} allows to relax this condition to  $d(z,S^c)\ge M \sqrt{\log n/n}$ for a potential $Q$ that is $\mathcal{C}^\infty$ on $\{Q<\infty\}$.
We also point out that the argument can be extended to potentials $Q$ that are $\mathcal{C}^\infty$ near the droplet (see \cite[Appendix A.2]{AmHeMa2}), although we will not invoke this fact.

Finally, given a point configuration $\{z_j\}_{j=1}^n$ we define the weighted Lagrange polynomials for $1\le j \le n$ by
\begin{align}\label{eq_lp}
	\ell_j(z)= \prod_{k\neq j}\frac{z-z_k}{z_j-z_k} e^{-n(Q(z)-Q(z_j))/2}.
\end{align}
We will use the following variant of Lemma \ref{lemma_quote}, which follows by inspecting the proofs of \cite[Theorem~3]{Am} and \cite[Theorems~1.1 and 5.1]{amro20}; see also
\cite[Section 3, Remark]{amro20}.

\begin{lemma}\label{lemma_quote_2}
Assume that $Q$ satisfies Conditions \ref{c1} to \ref{c4} and $\beta_n\ge c \log n$ and fix a neighborhood $U$ of $S$ where \ref{cau} is satisfied.
Then there exist positive constants $A=A(c,Q)$ and $n_0=n_0(c,Q)$ such that for
any $k \in \mathbb{N}$ and every $n \geq n_0$ the following hold with probability $\ge 1-n^{-k}$:
\begin{itemize}
\item[(i)] The points $\{z_j\}_j$ are contained in $U$ and $A^{-k}/\sqrt{n}$-separated:
\begin{align}\label{sep2}
	\sqrt{n} \min_{1 \leq j \not= k \leq n} |z_j-z_k| \geq A^{-k}.
\end{align}
\item[(ii)] The Lagrange polynomials \eqref{eq_lp} are bounded as follows:
\begin{align}
	\label{elag}
	\max_{1\le j\le n}\|\ell_j\|_\infty\le A^k.
\end{align} 
\end{itemize}
\end{lemma}

\subsection{Discrepancy }
	
Theorem~\ref{cdisbul} will be deduced from the following more general result for compact observation windows with regular boundary. Recall the notation for the equilibrium measure \eqref{eq_equil}.

\begin{theorem}\label{tdisbul2}
	Assume $Q\in \mathcal{C}^2$ and satisfies Conditions \ref{c1} to \ref{c4}, $\beta_n\ge c \log n$ and $\Omega\subseteq \C$ is a compact domain with regular boundary at scale $\eta>0$. Then there exist deterministic constants $M=M(Q)$, $n_0=n_0(c,Q)$ and $C=C(c,Q)$ such that for every $n\ge n_0$ and $k\in\N$, the following holds
	with $\pp_n^{\beta_n}$ probability $\ge 1- n^{-k}$: For every $p_n\in S$ with
	\begin{align}\label{hhh}
	d(p_n+\Omega/\sqrt{n},S^c)\ge M n^{-1/2}\log n
	\end{align}
	we have
	\begin{align}\label{eq_ggg}
		\left|\#\left[p_n+\Omega/\sqrt{n}\right]- n \, \mu(p_n+\Omega/\sqrt{n}) \right| \le C^k \cdot \frac{\hh(\partial\Omega)}{ \cu} \cdot \Big(1+\frac{1}{\eta}\Big),
	\end{align}
	where $\mu$ is the equilibrium measure associated with $Q$, cf. \eqref{eq_equil}.
	\end{theorem}
	\begin{proof}
Fix a compact domain $\Omega$ and write
\begin{align*}
		\Omega_n &=p_n+\tfrac{1}{\sqrt{n}}\Omega  & \text{and}&   &  N_n&=\#[\Omega_n].
\end{align*}
Since the left-hand side of \eqref{eq_ggg} is $\lesssim n$, we can assume that $n\geq \hh(\partial\Omega)$ -- cf. \eqref{eq_cubb}. In particular, from this assumption and the isoperimetric inequality,
\begin{align}\label{enbi}
	|\Omega_n|= \frac{|\Omega|}{n}\lesssim \frac{\hh(\partial\Omega)^2}{n} \lesssim \hh(\partial\Omega).
\end{align}
We invoke Lemma \ref{lemma_quote_2} and adopt its notation. We fix a sample of the Coulomb gas $\{z_j\}_1^n \subset U$ satisfying \eqref{sep2} and \eqref{elag}.  We also set
\begin{align*}
	s:= \min\Big\{1, \sqrt{n} \min_{1 \leq j \not= k \leq n} |z_j-z_k|\Big\},
\end{align*}
which satisfies $s \geq A^{-k}$ by \eqref{sep2}, assuming, as we may, that $A \geq 1$. Taking $n_0$ large, we can further assume that $B_{s/\sqrt{n}}(z_j)\subseteq U$ and
$\Omega_n \subset S$.

Consider the concentration operator $T_{n}:\ww_{ n}\to \ww_{ n}$ by
\[T_{n} f=P_{\ww_{n}}(\chi_{\Omega_n} f),\]
and estimate
\begin{align}\label{etri}
	|N_n -   n \, \mu(\Omega_n)|\le 
	|N_n -   \tr(T_n)|
	+|\tr(T_n) -   n \, \mu(\Omega_n)|.
\end{align}
Regarding the second term on the right-hand side, note that
\begin{align*}
	\tr(T_n) -   n \, \mu(\Omega_n)
	&= \int_{\Omega_n} K_n(z,z)  - n\Delta Q(z)  \,dA(z)
\end{align*}
By Lemma \ref{lpun}, \eqref{hhh} and increasing $n_0$ if necessary,
\[|K_n(z,z)  - n\Delta Q(z)|\lesssim 1, \qquad z\in \Omega_n, n\ge n_0.\]
So, by \eqref{enbi} and \eqref{eq_cubb},
\begin{align}\label{etra}
	|\tr(T_n) -   n \, \mu(\Omega_n)|\lesssim |\Omega_n|\lesssim \hh(\partial\Omega) \lesssim \frac{\hh(\partial\Omega)}{ \cu} \cdot \Big(1+\frac{1}{\eta}\Big).
\end{align}
	
It only remains to bound the first term in \eqref{etri}. We proceed as in the proof of \cite[Theorem~1]{LeCe} which was in turn inspired by \cite{NiOl}. We first note the identity
\[K_n(z,w)=\sum_{j=1}^n \ell_j(z) K_n(z_j,w), \qquad z,w\in \C,\]
where $\ell_j$'s are the weighted Lagrange polynomials given by \eqref{eq_lp},
and write
\begin{align*}
	N_n -   \tr(T_n)&=\sum_{z_j\in \Omega_n} \ell_j(z_j) -\int_{\Omega_n} K_n(w,w) dA(w)
	\\ &=\sum_{z_j\in \Omega_n} \int_\C \ell_j(w) K_n(z_j,w) dA(w) -\int_{\Omega_n} \sum_{j=1}^n \ell_j(w) K_n(z_j,w) dA(w)
	\\ &= \int_{\Omega_n^c} \sum_{z_j\in \Omega_n} \ell_j(w) K_n(z_j,w) dA(w) -\int_{\Omega_n} \sum_{z_j\in \Omega_n^c} \ell_j(w) K_n(z_j,w) dA(w).
\end{align*}
We shall use the following estimate for weighted polynomials:
\begin{align}
	\label{lamro}
	|f(z_0)|\lesssim 	\frac{n}{s^2} \int_{B_{s/\sqrt{n}}(z_0)} |f(z)| dA(z),
	\qquad 
	f\in \ww_n, \quad B_{s/\sqrt{n}}(z_0)\subseteq U;
\end{align}
which follows from the fact that $Q$ is smooth and bounded on $U$ \cite{SaTo}; see \cite[Lemma~2.1]{amro20} for a direct proof. 

We now invoke \eqref{sep2} and \eqref{elag}, and apply \eqref{lamro}
to the reproducing kernels, together with the fact that  $B_{s/\sqrt{n}}(z_j)\subseteq U$ are disjoint:
\begin{align}
	&|N_n -   \tr(T_n)|\notag
	\\
	&\qquad \lesssim A^k \int_{\Omega_n^c} \sum_{z_j\in \Omega_n} |K_n(z_j,w)| dA(w) + A^k\int_{\Omega_n} \sum_{z_j\in \Omega_n^c} |K_n(z_j,w)| dA(w)\notag
	\\
	&\qquad \lesssim \frac{A^k n}{s^2} 
		\Big(\int_{\Omega_n^c} \int_{\Omega_n+B_{s/\sqrt{n}}(0)} |K_n(z,w)| dA(z,w) 
		+\int_{\Omega_n} \int_{\Omega_n^c+B_{s/\sqrt{n}}(0)} |K_n(z,w)| dA(z,w)\Big)\notag
	\\ &\qquad \lesssim \label{econ}
			A^{3k} n
			\int_{\Omega_n} \int_{\Omega_n^c} |K_n(z,w)| dA(w,z)
			 + A^{3k} n \int_{\partial\Omega_n+B_{s/\sqrt{n}}(0)} \int_\C |K_n(z,w)| dA(w,z).
\end{align}
We treat each integral separately and start with the first one. 
Let $\varepsilon>0$ and $h\in L^1(\C)$ be as in Lemma~\ref{efine} --- and enlarge $n_0$ if needed to satisfy the conclusions. Since $d(\Omega_n,S^c)\ge n^{-1/2} M \log n $,
\begin{align}\label{eext}
	\begin{aligned}
	n \int_{\Omega_n} \int_{\Omega_n^c\cap \{w:|z-w|\ge d(z,S^c)\}} |K_n(z,w)|  dA(w,z)
	&\lesssim n^2 |\Omega_n|  e^{-\varepsilon M \log n } +  |\Omega_n| \|h\|_1
	\\ &= |\Omega_n| (e^{(2-\varepsilon M) \log n } +  \|h\|_1) \lesssim  \hh(\partial\Omega),
	\end{aligned}
\end{align}
where in the last step we choose $M>2 \varepsilon^{-1}$ and use \eqref{enbi}.
		
Regarding the region where $|z-w|\le d(z,S^c)$, using Lemma~\ref{efine} again we have
\begin{align*}
	n \int_{\Omega_n}  \int_{\Omega_n^c\cap \{w:|z-w|\le d(z,S^c)\}} |K_n(z,w)| dA(w)  dA(z)
	&\lesssim n^2 \int_{\Omega_n}\int_{\Omega_n^c} e^{-\varepsilon \sqrt{n}  |z-w|} dA(w)  dA(z)
	\\& = \int_{\Omega} \int_{\Omega^c}  e^{-\varepsilon  |z-w|} dA(w) dA(z)\lesssim \hh( \partial\Omega), \notag
\end{align*}
where the last step follows from an elementary estimate; see, e.g., \cite[Lemma~3.2]{AbGrRo}. Combining this with \eqref{eext} and using \eqref{eq_cubb}, we get
\begin{align}
	\label{e1}
	n\int_{\Omega_n} \int_{\Omega_n^c} |K_n(z,w)| dA(w)  dA(z)
	&\lesssim \hh( \partial\Omega)\lesssim \frac{\hh(\partial\Omega)}{ \cu} \cdot \Big(1+\frac{1}{\eta}\Big).
\end{align}
For the second integral in \eqref{econ}, we proceed similarly as before. Writing $E_{n,s}=\partial\Omega_n +B_{s/\sqrt{n}}(0)$ 
and since
$d(E_{n,s},S^c)\ge n^{-1/2} (M \log n -s) $, 
\begin{multline}
	\label{eext1}
	n \int_{E_{n,s}} \int_{\{w:|z-w|\ge d(z,S^c)\}} |K_n(z,w)|  dA(w) dA(z)
	\lesssim  |E_{n,s}| (n^2 e^{-\varepsilon M \log n } +  \|h\|_1)
	\\ = n^{-1}|\partial\Omega+B_{s}(0)| ( e^{(2-\varepsilon M) \log n } +  \|h\|_1) \lesssim \frac{\hh(\partial\Omega)}{ \cu}\Big(1+\frac{1}{\eta}\Big), 
\end{multline}
where in the last step we used Corollary~\ref{coar} and that $s\le 1$.
		
Regarding the region where $|z-w|\le d(z,S^c)$,
\begin{align*}
	n \int_{E_{n,s}}  \int_{\{w:|z-w|\le d(z,S^c)\}} |K_n(z,w)| dA(w)  dA(z)
	&\lesssim n^2 \int_{E_{n,s}}\int_{\C} e^{-\varepsilon \sqrt{n}  |z-w|} dA(w)  dA(z)
	\\& = \int_{\partial\Omega+B_{s}(0)} \int_{\C}  e^{-\varepsilon  |z-w|} dA(w) dA(z) \notag
	\\ &\lesssim |\partial\Omega+B_{s}(0)|\lesssim \frac{\hh(\partial\Omega)}{ \cu}\Big(1+\frac{1}{\eta}\Big), \notag
\end{align*}
where again the last step follows from Corollary~\ref{coar}. 
			
Combining this with \eqref{eext1}, we conclude that
\begin{align*}
	n\int_{E_{n,s}} \int_{\C} |K_n(z,w)| dA(w)  dA(z)
	&\lesssim \frac{\hh(\partial\Omega)}{ \cu}\Big(1+\frac{1}{\eta}\Big),
\end{align*}
which together with \eqref{econ} and \eqref{e1}, gives
\begin{align*}
	|N_n -   \tr(T_n)| 
	& \lesssim A^{3k} \frac{\hh(\partial\Omega)}{ \cu} \Big(1+\frac{1}{\eta}\Big).
\end{align*}
Finally, the previous estimate, together with
\eqref{etri} and \eqref{etra} yield \eqref{eq_ggg}.
\end{proof}
	
\begin{proof}[Proof of Theorem~\ref{cdisbul}]
Since $\Omega$ has a connected boundary, by Lemma~\ref{rrb}, $\Omega$ has regular boundary at scale
\begin{align*}
\eta=\hh(\partial\Omega),
\end{align*}
and
\begin{align}\label{eq2}
1\le \cu\le 2.
\end{align} 
By the isoperimetric inequality,
\begin{align}\label{eq1}
\eta=\hh(\partial\Omega) \geq 2 \sqrt{\pi |\Omega|} \geq 1.
\end{align}
Let $p \in \mathbb{C}$ satisfy $d(p,S^c)\ge M n^{-1/2}\log n$ with a constant $M$ to be specified. Since $\Omega$ is bounded,
\begin{align}\label{eq_xyz}
d(p+\Omega/\sqrt{n},S^c)\ge \frac{M}{2} n^{-1/2}\log n
\end{align}
for sufficiently large $n$. We invoke Theorem~\ref{tdisbul2} with $k=2$, and let $M$ be twice the corresponding constant in the statement of Theorem~\ref{tdisbul2}. By \eqref{eq2} and \eqref{eq1}, the factor $(1+1/\eta)/\cu$ in \eqref{eq_ggg} is $\lesssim 1$. We conclude that for all sufficiently large $n$,
\begin{align*}
\pp_n^{\beta_n}\Big(\Big\{\sup_{\substack{p\in S,\\ d(p,S^c)\ge M n^{-1/2}\log n}}\left|\#\left[p+\Omega/\sqrt{n}\right]- n \, \mu(p+\Omega/\sqrt{n}) \right|  > C \hh(\partial\Omega)\Big\}\Big)\le n^{-2}.
\end{align*}
By the Borell-Cantelli lemma, almost surely,
\[\limsup_{n\to\infty} \sup_{\substack{p\in S,\\ d(p,S^c)\ge M n^{-1/2}\log n}}\left|\#\left[p+\Omega/\sqrt{n}\right]- n \, \mu(p+\Omega/\sqrt{n}) \right|  \le C \hh(\partial\Omega). \]
Finally, \eqref{eq_disbul} follows by noting that, uniformly for $p\in S$,
\begin{align*}
	\left|n \, \mu(p+\Omega/\sqrt{n}) - \Delta Q (p)\frac{|\Omega|}{\pi}\right|&\le n\int_{p+\Omega/\sqrt{n}} |\Delta Q (z)- \Delta Q (p)| dA(z)
	\\
	&\lesssim n \int_{p+\Omega/\sqrt{n}} |z-p| dA(z) 
	\\ &\lesssim |\Omega| \frac{\text{diam}(\Omega)}{\sqrt n} \xrightarrow{n\to\infty} 0,
\end{align*}
where the uniform Lipschitz estimate $|\Delta Q(\cdot) -\Delta Q(p)|\lesssim |\cdot-p|$ on $p + \Omega/\sqrt{n}$ follows from the fact $p + \Omega/\sqrt{n} \subset S$ for all sufficiently large $n$, cf. \eqref{eq_xyz}, and $Q$ is real-analytic on a neighborhood of $S$.
\end{proof}

\end{document}